\documentclass[10pt]{article}%

\usepackage[utf8]{inputenc}

\usepackage{amsmath}

\usepackage{amsfonts}
\usepackage{mathrsfs}
\usepackage{amssymb, color}
\usepackage[linkcolor=black,anchorcolor=black,citecolor=black]{hyperref}
\usepackage{graphicx}
\numberwithin{equation}{section}
\usepackage[body={15.5cm,21cm}, top=3cm]{geometry}%
\providecommand{\U}[1]{\protect\rule{.1in}{.1in}}
\providecommand{\U}[1]{\protect \rule{.1in}{.1in}}
\newtheorem{theorem}{Theorem}[section]

\newtheorem{corollary}[theorem]{Corollary}

\newtheorem{definition}[theorem]{Definition}

\newtheorem{lemma}[theorem]{Lemma}

\newtheorem{proposition}[theorem]{Proposition}
\newtheorem{remark}[theorem]{Remark}

\newtheorem{assumption}[theorem]{Assumption}
\newenvironment{proof}[1][Proof]{\noindent \textbf{#1.} }{\  \rule{0.5em}{0.5em}}
\DeclareMathOperator*{\esssup}{ess\,sup}
\DeclareMathOperator*{\essinf}{ess\,inf}

\def \E{\mathsf{E}}

\def \P{\mathsf{P}}

\usepackage [numbers,sort&compress] {natbib}

\def \E{\mathsf{E}}
\def \P{\mathsf{P}}

\begin{document}
	\title{Conditional Reflected Backward Stochastic Differential Equations with Two Barriers}
	\author{ 	Hanwu Li\thanks{Research Center for Mathematics and Interdisciplinary Sciences, Shandong University, Qingdao 266237, Shandong, China. lihanwu@sdu.edu.cn.}
	\thanks{Frontiers Science Center for Nonlinear Expectations (Ministry of Education), Shandong University, Qingdao 266237, Shandong, China.}}
	\date{}
	\maketitle
	
	\begin{abstract}
        In this paper, we study the doubly conditional  reflected backward stochastic differential equations (BSDEs), where  constraints are made on the conditional expectation of the first component of the solution with respect to a general subfiltration. With the help of the Skorokhod problem on a time-dependent interval and the Dynkin game in a general framework, we establish the existence and uniqueness result under the Mokobodski condition for the obstacles. The relation between the conditional expectation of the solution and the value function of a certain Dynkin game with partial information is obtained. As a by-product, we obtain a weaker version of the comparison theorem. Finally, we provide an application to the starting and stopping problem in reversible investments under partial information. 
	\end{abstract}


    \textbf{MSC-classification}: 60H10
	
\section{Introduction}

Reflected BSDEs were introduced by El Karoui et al. in \cite{EKPPQ} taking the following form
\begin{align}\label{intro1}
\begin{cases}
Y_t=\xi+\int_t^T f(s,Y_s,Z_s)ds-\int_t^T Z_s dB_s+(K_T-K_t),\\
Y_t\geq S_t, \ t\in[0,T],\\
K_0=0, \ K \textrm{ is a nondecreasing process such that } \int_0^T (Y_s-S_s)dK_s=0.
\end{cases}
\end{align}
In \cite{EKPPQ}, the authors establish the well-posedness of the solution $(Y,Z,K)$ to the above equation as well as its connection to the optimal stopping problems and related partial differential equations (PDEs) with obstacles. When the solution $Y$ is forced to stay between two prescribed obstacles $L,U$, the equation is called the doubly reflected BSDE first investigated by Cvitani\'{c} and Karatzas in \cite{CK}. In order to obtain the square integrable solutions to the doubly reflected BSDEs, the obstacles $L,U$ are assumed to satisfy the so-called Mokobodski condition, i.e., we may place the discrepancy of two nonnegative supermartingales between two obstacles (see, e.g., \cite{CK,HL,PX}). Due to the importance in both theoretical analysis and practical applications, many extension works on doubly reflected BDSEs have been studied. To name a few, we refer to the papers \cite{CM,DQS,EH,FS20,GIOQ,HH,HHO,K2021} and the references therein. 

In 2018, Briand, Elie and Hu \cite{BEH} proposed a new kind of constrained BSDEs, called mean reflected BSDEs, where the restriction for $Y$ is not given on its path but on the distribution, which is written as follows
\begin{align*}
\E[l(t,Y_t)]\geq 0, \ t\in[0,T].
\end{align*}
Here, $l$ is a given bi-Lipshcitz loss function. To ensure the uniqueness of the solution, different from the classical reflected BSDEs, the process $K$ should be deterministic and satisfies the Skorokhod condition 
\begin{align*}
    \int_0^T\E[l(t,Y_t)]dK_t=0.
\end{align*}
The mean reflected BSDE gives the superhedging price for a contingent claim under a running risk manangement constraint. Since then, the mean reflected problem has attracted considerable attention. We may refer to \cite{BCGL,FS21} for mean reflected SDEs, \cite{BH} for particles systems, \cite{DEH} for the mean-field case, \cite{CZ,HHLLW} for the case of non-Lipschitz driver, \cite{LW} for the $G$-expectation framework, \cite{QW} for the multi-dimensional case and \cite{FS,L} for the case of double mean reflections. 

Recently, motivated by the pricing for American options with partial information and  the recursive reflected utility maximization problem with partial information, Hu, Huang and Li \cite{HHL} introduced the so-called conditional reflected BSDEs, where the reflection barrier is defined via a conditional expectation operator on a general subfiltration $\mathbb{G}=\{\mathcal{G}_t\}_{t\in [0,T]}$. More precisely, the constraint is given as 
\begin{align*}
    \E[Y_t-S_t|\mathcal{G}_t]\geq 0.
\end{align*}
It is worth pointing out that the partial information features in mathematical finance have been extensively investigated (e.g., \cite{BDL,Lindensjo,SH,ZXZ}). Clearly, the classical reflected BSDE and the mean reflected BSDE (with linear loss function $l$) serve as two special examples of conditional reflected BSDE. The first  corresponds to the case that the subfiltration turns into the filtration generated by Brownian motion $B$ and the latter coincides with the case that the subfiltration degenerates into the deterministic scenario. 

In the present paper, we consider the conditional reflected BSDEs with two barriers, that is, the first component $Y$ of the solution is subject to a constraint taking the following form
\begin{align*}
    \E[L_t|\mathcal{G}_t]\leq \E[Y_t|\mathcal{G}_t]\leq \E[U_t|\mathcal{G}_t].
\end{align*}
To fulfill this condition, the term $K$ should consist of the forces aiming to push the solution upward (denoted by $K^+$) and to pull the solution downward (denoted by $K^-$). Moreover, both forces need to behave in a minimal way such that the following Skorokhod condition are satisfied
\begin{align*}
    \int_0^T \E[Y_t-L_t|\mathcal{G}_t]dK_t^+=\int_0^T \E[U_t-Y_t|\mathcal{G}_t]dK^-_t=0.
\end{align*}
In order to guaranty the uniqueness of the solution, $K^+$ and $K^-$ are required to be $\mathbb{G}$-adpated as explained in \cite{BEH}. Similar with the case of single conditional reflection, the doubly conditional  reflected BSDEs subsume classical doubly reflected BSDE and the doubly mean reflected BSDE (with linear loss function) as its two special and extreme cases. Moreover, if the upper obstacle $U\equiv +\infty$, the doubly conditional reflected BSDE degenerates into the conditional reflected BSDE studied in \cite{HHL}. 

It should be pointed out that the proof for the well-posedness of doubly mean reflected BSDEs heavily relies on the fact that $K$ is deterministic. Therefore, the techniques using the backward Skorokhod problem are no longer valid for the conditional reflected case. In order to prove the uniqueness, a key point is to establish the representation for $K$ (i.e., Proposition \ref{repK}) using the Skorokhod problem on time-dependent interval (see \cite{BKR,Slaby}), which will help to obtain certain a priori estimates for doubly conditional reflected BSDEs. The existence can be obtained by a contraction mapping argument. The building block is the construction of the solution to the case that the driver $f$ does not depend on $(y,z)$, which can be established by using the Dynkin game and the optimal stopping time problem with partial information (see \cite{KQ,KQdC}). Another frequently used method to construct the solutions to (mean) reflected BSDEs is approximation via penalization. For the doubly conditional reflected case, this method is still valid. It is worth pointing out that the approximation sequence is a family of penalized conditional expectation BSDEs (see \cite{Li}). Due to the weak constraint condition, the penalization method needs some additional assumptions for the driver $f$ and the obstacles $L,U$. The advantage of this construction is that it does not need any continuity assumption for subfiltration $\mathbb{G}$.

One of the most important properties for reflected BSDEs is the comparison theorem. A natural question is that if the pointwise comparison property still holds for the conditional reflected BSDEs. Unfortunately, since the constraint is made on conditional expectation but not pointwise, the comparison theorem only holds under some typical structure of the parameters (see  Example 3.2 in \cite{HHLLW} as a counterexample for the general case).  Motivated by Corollary 3.3 in \cite{HHL}, we compare the conditional expectation of the solution with respect to the partial information. To this end, we first establish the connection between the conditional expectation of the solution and the value function of a certain Dynkin game with partial information. Then, the comparison property is equivalent to the comparison for the reward processes of the associated Dynkin games.  

As one application of the doubly conditional reflected BSDEs, we consider the starting and stopping problem in reversible investment under partial information $\mathbb{G}$. Roughly speaking, we aim to find a sequence of $\mathbb{G}$-stopping times  where the agent should decide when to stop the production and to resume it successively in order to maximize the overall profit. One important observation is that the solution to a doubly conditional reflected BSDEs can be represented as the difference between two solutions to certain conditional reflected BSDEs. We show that one of the solutions is indeed the value function of the starting and stopping problem. The optimal strategy can be constructed by the interpretation of the solution to conditional reflected BSDEs.

This paper is organized as follows. We first introduce some preliminaries about Dynkin games in a general framework in Section 2. In Section 3, we formulate the conditional doubly reflected BSDEs in details and establish the existence and uniqueness result. The properties of the solutions to conditional doubly reflected BSDEs are given in Section 4. In the last section, we give some application of doubly conditional reflected BSDE in the reversible investment problem. In the Appendix, we construct the solution by a penalization method. 

\section{Dynkin games in a general framework}

In this section, we first recall some basic notations and results about  Dynkin games in a general framework (see \cite{KQdC}). Let $(\Omega,\mathcal{F},\mathbb{F},\P)$ be a probability space equipped with a filtration $\mathbb{F}=\{\mathcal{F}_t\}_{t\in[0,T]}$ satisfying the usual conditions of right continuity and completeness. Let $\mathcal{T}^\mathbb{F}$ be the collections of all $\mathbb{F}$-stopping times taking values in $[0,T]$. For any $S,S'\in\mathcal{T}^\mathbb{F}$ with $S\leq S'$ a.s., we denote by $\mathcal{T}^{\mathbb{F}}_S$ (resp., $\mathcal{T}^\mathbb{F}_{S,S'}$) the class of stopping times $\tau\in\mathcal{T}^\mathbb{F}$ with $\tau\geq S$ a.s. (resp., $S\leq \tau\leq S'$ a.s.).  We always omit the superscript $\mathbb{F}$ if the filtration is obvious. 

\begin{definition}
    A family of $\bar{\mathbb{R}}$-valued random variables $\{\phi(\theta),\theta\in\mathcal{T}\}$ is said to be admissible if it satisfies the following conditions:
    \begin{itemize}
        \item[(1)] for all $\theta\in\mathcal{T}$, $\phi(\theta)$ is an $\mathcal{F}_\theta$-measurable random variable,
        \item[(2)] for all $\theta,\theta'\in\mathcal{T}$, $\phi(\theta)=\phi(\theta')$ a.s. on $\{\theta=\theta'\}$.
    \end{itemize}
    The set of admissible families is denoted by $\mathcal{A}$. A family $\phi\in\mathcal{A}$ is said to be integrable if, for each $\theta\in\mathcal{T}$, $\phi(\theta)$ is integrable.
\end{definition}

For $\phi,\phi'\in\mathcal{A}$, we write $\phi\leq \phi'$ if,  for each $\theta\in\mathcal{T}$, $\phi(\theta)\leq \phi'(\theta)$ a.s. For $\phi\in\mathcal{A}$, $\phi^+$ (resp., $\phi^-$) denotes the family $\{(\phi(\theta))^+,\theta\in\mathcal{T}\}$ (resp., $\{(\phi(\theta))^-,\theta\in\mathcal{T}\}$). We define the following subsets of $\mathcal{A}$:
\begin{align*}
    &\mathcal{S}=\left\{\phi\in\mathcal{A}:\E\left[\esssup_{\theta\in\mathcal{T}}|\phi(\theta)|\right]<\infty\right\},\\
    \mathcal{S}^+=&\{\phi\in\mathcal{A}:\phi^+\in\mathcal{S}\}, \ \  \ \ \ \mathcal{S}^-=\{\phi\in\mathcal{A}:\phi^-\in\mathcal{S}\}.
\end{align*}

\begin{definition}
    An admissible family $\phi$, such that $\phi^-$ is integrable,  is said to be a supermartingale family (resp., martingale family) if for any $\theta,\theta'\in\mathcal{T}$ such that $\theta\geq \theta'$ a.s.,
    \begin{align*}
        \E[\phi(\theta)|\mathcal{F}_{\theta'}]\leq \phi(\theta') \ \textrm{a.s.} \ (\textrm{resp., } \E[\phi(\theta)|\mathcal{F}_{\theta'}]= \phi(\theta') \ \textrm{a.s.}).
    \end{align*}
\end{definition}

\begin{definition}
    For each $\phi\in\mathcal{A}$, the smallest supermartingale family which is no smaller than $\phi$ is called the Snell envelope family of $\phi$, and is denoted by $\mathcal{R}(\phi)$.
\end{definition}

\begin{remark}
    Let $\phi\in\mathcal{S}^-$ be the reward family. For each $\theta\in\mathcal{T}$, the value function of the optimal stopping problem at time $\theta$ is defined by 
    \begin{align*}
v^\phi(\theta):=\esssup_{\tau\in\mathcal{T}_\theta}\E[\phi(\tau)|\mathcal{F}_\theta].
    \end{align*}
    It is shown in \cite{KQ} that the value function family $v$ associated with $\phi$ is the smallest supermartingale family which dominated $\phi$. That is, if $\phi\in\mathcal{S}^-$, we have $\mathcal{R}(\phi)=v^\phi$.
\end{remark}

In the following of this section, let $\xi,\zeta$ be two integrable families such that $\xi\in\mathcal{S}^-$, $\zeta\in\mathcal{S}^+$ and $\xi(T)=\zeta(T)=0$ a.s. For each $\theta\in\mathcal{T}$, the lower value function $\underline{V}$ and the upper value function $\overline{V}$ of the Dynkin game at time $\theta$ are defined by
\begin{align*}
    &\underline{V}(\theta):=\esssup_{\tau\in\mathcal{T}_\theta}\essinf_{\sigma\in\mathcal{T}_\theta}I_\theta(\tau,\sigma),\\
    &\overline{V}(\theta):=\essinf_{\sigma\in\mathcal{T}_\theta}\esssup_{\tau\in\mathcal{T}_\theta}I_\theta(\tau,\sigma),
\end{align*}
where the criterion $I_\theta(\tau,\sigma)$ at time $\theta$ for a strategy $(\tau,\sigma)$ is defined by
\begin{align*}
    I_\theta(\tau,\sigma):=\E[\xi(\tau) I_{\{\tau\leq \sigma\}}+\zeta(\sigma) I_{\{\sigma<\tau\}}|\mathcal{F}_\theta].
\end{align*}

\begin{remark}
    Since the criterion does not depend on the terminal value $\zeta(T)$, it is not restrictive to assume that $\xi(T)=\zeta(T)$. Moreover, the assumption $\xi(T)=\zeta(T)=0$ is no more restrictive. Actually, give two integrable families $\xi,\zeta$ with $\xi(T)=\zeta(T)$, we define the integrable families $\xi',\zeta'$ as follows
    \begin{align*}
        \xi'(\theta):=\xi(\theta)-\E[\xi(T)|\mathcal{F}_\theta], \ \zeta'(\theta):=\zeta(\theta)-\E[\xi(T)|\mathcal{F}_\theta], \ \forall \theta \in\mathcal{T}.
    \end{align*}
    It is easy to check that $\xi'(T)=\zeta'(T)=0$ and the criterion $I_\theta(\tau,\sigma)$ associated with $\xi,\zeta$ can be rewritten as
    \begin{align*}
    I_\theta(\tau,\sigma)=\E[\xi'(\tau) I_{\{\tau\leq \sigma\}}+\zeta'(\sigma) I_{\{\sigma<\tau\}}|\mathcal{F}_\theta]+\E[\xi(T)|\mathcal{F}_\theta].
\end{align*}
Therefore, solving the Dynkin game with reward families $\xi,\zeta$ degenerates into solving the Dynkin game with reward families $\xi',\zeta'$.

A sufficient condition to ensure that $\xi'\in\mathcal{S}^-$ and $\zeta'\in\mathcal{S}^+$ is that $\xi\in\mathcal{S}^-$, $\zeta\in\mathcal{S}^+$ and $\xi(T)$ is $p$-integrable for some $p>1$.
\end{remark}

\begin{theorem}\label{theorem2.3}
    There exist two nonnegative supermartingale families $J$ and $J'$ which satisfy $J=\mathcal{R}(J'+\xi)$ and $J'=\mathcal{R}(J-\zeta)$. Moreover, $J$ and $J'$ can be constructed in a minimal fashion. That is, if $\bar{J},\bar{J}'$ are two supermartingale families satisfying $\bar{J}=\mathcal{R}(\bar{J}'+\xi)$ and $\bar{J}'=\mathcal{R}(\bar{J}-\zeta)$, then we have $J\leq \bar{J}$ and $J'\leq \bar{J}'$.
\end{theorem}

\begin{proposition}\label{proposition2.5}
    The condition $J(0)<+\infty$ is equivalent to the condition $J'(0)<+\infty$. Moreover, if $J(0)<+\infty$, the family of random variables $Y$ given by $Y:=J-J'$ is well-defined and satisfies
    \begin{align*}
        \xi\leq Y\leq \zeta.
    \end{align*}
\end{proposition}

Recall that the Mokobodski condition (see \cite{CK,HL,PX}) for $\xi,\zeta$ amounts to say the existence of two nonnegative a.s. finite supermartingales $H$ and $H'$ such that $\xi\leq H-H'\leq \zeta$. Hence, Proposition \ref{proposition2.5} indicates that if $J(0)<+\infty$, then the Mokobodski condition holds. Actually, these two conditions are equivalent, which is shown in the following proposition.

\begin{proposition}\label{proposition5.2}
    The condition $J(0)<+\infty$ (or equivalently $J'(0)<+\infty$) is equivalent to the Mokobodski condition for $\xi,\zeta$, i.e., there exist two nonnegative a.s. finite supermartingales $H$ and $H'$ such that $\xi\leq H-H'\leq \zeta$. Moreover, if $H^2,(H')^2\in \mathcal{S}$, we have $J^2,(J')^2\in\mathcal{S}$.
\end{proposition}

In order to make sure the Dynkin game is fair, i.e., the upper value and lower value are equal, we need to propose the following regularity condition for the reward family.

\begin{definition}
    A family $\phi\in\mathcal{S}^-$ is said to be right-(resp. left-)upper semicontinuous in expectation along stopping times (RUSCE (resp. LUSCE)) if for any $\theta\in\mathcal{T}$ and for any sequences of stopping times $\{\theta_n\}_{n\in\mathbb{N}}$ such that $\theta_n\downarrow \theta$ (resp. $\theta_n\uparrow \theta$), we have
    \begin{align*}
        \E[\phi(\theta)]\geq \limsup_{n\rightarrow \infty}\E[\phi(\theta_n)].
    \end{align*}
    Moreover, $\phi$ is said to be USCE if it is both RUSCE and LUSCE.
\end{definition}

\begin{theorem}\label{theorem3.6}
    Suppose that $J(0)<+\infty$ and the families $\xi$ and $-\zeta$ are RUSCE. Then, the Dynkin game is fair and the common value function is equal to $Y$. That is, for any $\theta\in\mathcal{T}$, we have
    \begin{align*}
        Y(\theta)=\underline{V}(\theta)=\overline{V}(\theta) \textrm{ a.s.}
    \end{align*}
\end{theorem}

\begin{remark}\label{r2.11}
Suppose that the reward are given by two progressive processes $\xi=\{\xi_t\}_{t\in[0,T]},\zeta=\{\zeta_t\}_{t\in[0,T]}$ such that $\xi_T=\zeta_T=0$ and 
\begin{align*}
    \E\left[\esssup_{t\in[0,T]}\xi^-_t\right]<+\infty,  \ \E\left[\esssup_{t\in[0,T]}\zeta^+_t\right]<+\infty.
\end{align*}  
All the results in this section still hold. Besides, all the admissible families, e.g., $\underline{V},\overline{V},J,J',Y$ can be aggregated into the corresponding progressive processes. Moreover, if the reward processes are continuous, all the resulting processes are continuous. 
\end{remark}

\section{Doubly conditional reflected BSDEs}

We are given a finite time horizon $T>0$. Let $B$ be a $d$-dimensional standard Brownian motion defined on  a  probability space $(\Omega,\mathcal{F},\P)$. We denote by $\mathbb{F}=\{\mathcal{F}_t\}_{t\in[0,T]}$ be the complete filtration generated by $B$. Consider a subfiltration $\mathbb{G}=\{\mathcal{G}_t\}_{t\in[0,T]}$ of $\mathbb{F}$, i.e., for any $t\in[0,T]$, $\mathcal{G}_t\subset \mathcal{F}_t$. Throughout this paper, we assume that $\mathbb{G}$ satisfies the following assumptions.

\begin{assumption}\label{assG}
    \begin{itemize}
        \item[(i)] $\mathbb{G}$ is nondecreasing and right-continuous;
        \item [(ii)] $\mathbb{G}$ is left-quasi-continuous. That is, $\mathbb{G}$ is left-continuous along stopping times.
    \end{itemize}
\end{assumption}

\begin{remark}
    $\mathbb{F}$ and $\mathbb{H}=\{\mathcal{H}_t\}_{t\in[0,T]}$ are two trivial examples of subfiltration satisfying Assumption \ref{assG}, where  $\mathcal{H}_t=\mathcal{F}_0$, for any $t\in[0,T]$. $\mathbb{H}$ is called the deterministic scenario.
\end{remark}

We first introduce the following notations, which will be frequently used in this paper.
\begin{itemize}
\item $L^2(\mathcal{F}_t)$: the set of real-valued $\mathcal{F}_t$-measurable random variable $\xi$ such that $\E[|\xi|^2]<\infty$.
\item $\mathcal{S}^2(0,T;\mathbb{R})$: the set of real-valued $\mathbb{F}$-adapted continuous processes $Y$ on $[0,T]$ such that $$\E\left[\sup_{t\in[0,T]}|Y_t|^2\right]<\infty.$$
\item $\mathcal{H}^2(0,T;\mathbb{R}^d)$: the set of $\mathbb{R}^d$-valued $\mathbb{F}$-progressively measurable processes $Z$ such that $$\E\left[\int_0^T|Z_t|^2dt\right]<\infty.$$
\item $\mathcal{A}^2_\mathbb{G}(0,T;\mathbb{R})$: the set of $\mathbb{G}$-adapted nondecreasing processes $K\in \mathcal{S}^2(0,T;\mathbb{R})$ such that $K_0=0$.
\item $\mathcal{BV}^2_{\mathbb{G}}(0,T;\mathbb{R})$: the set of $\mathbb{G}$-adapted processes  $K\in \mathcal{S}^2(0,T;\mathbb{R})$ such that $K\equiv K^1-K^2$ with $K^i\in \mathcal{A}^2_{\mathbb{G}}(0,T;\mathbb{R})$, $i=1,2$.
\item $C[0,T]$: the set of continuous functions from $[0,T]$ to $\mathbb{R}$.
\item $I[0,T]$: the set of functions in $C[0,T]$ starting from the origin which is nondecreasing.
\end{itemize}

For simplicity, we always omit the brackets when there is no confusion. The main purpose of this paper is to study the doubly conditional  reflected BSDE  the following type
\begin{equation}\label{nonlinearyz}
\begin{cases}
Y_t=\xi+\int_t^T f(s,Y_s,Z_s)ds-\int_t^T Z_s dB_s+K_T-K_t, \\
\E[L_t|\mathcal{G}_t]\leq \E[Y_t|\mathcal{G}_t]\leq \E[U_t|\mathcal{G}_t], \\
K_t=K^+_t-K^-_t, \ K^+,K^-\in \mathcal{A}^2_{\mathbb{G}},\\
\int_0^T \E[Y_t-L_t|\mathcal{G}_t]dK_t^+=\int_0^T \E[U_t-Y_t|\mathcal{G}_t]dK^-_t=0.
\end{cases}
\end{equation}
The above doubly conditional reflected BSDE is determined by the following parameters: the terminal value $\xi$, the driver $f$, the lower obstacle $L$ and the upper obstacle $U$. We will make the following assumptions for the parameters.

\begin{itemize}
\item[(H1)]The driver $f$ is a map from $\Omega\times[0,T]\times \mathbb{R}\times\mathbb{R}^d$ to $\mathbb{R}$. For 
each fixed $(y,z)$, $f(\cdot,\cdot,y,z)$ is progressively measurable. There exists $\lambda>0$ such that for any $t\in[0,T]$ and any $y,y'\in\mathbb{R}$, $z,z'\in\mathbb{R}^d$
\begin{align*}
|f(t,y,z)-f(t,y',z')|\leq \lambda(|y-y'|+|z-z'|)
\end{align*}
and 
\begin{align*}
\E\left[\int_0^T |f(t,0,0)|^2dt\right]<\infty.
\end{align*}
\item[(H2)] The obstacles $L,U\in \mathcal{S}^2$ satisfy
\begin{align*}
    \inf_{(t,\omega)\in [0,T]\times \Omega}(U_t(\omega)-L_t(\omega))>0.
\end{align*}
\item[(H3)] The terminal value $\xi\in L^2(\mathcal{F}_T)$ and $\E[L_T|\mathcal{G}_T]\leq \E[\xi|\mathcal{G}_T]\leq \E[U_T|\mathcal{G}_T]$.
\end{itemize}

\begin{remark}
    (i) Suppose that $U\equiv +\infty$. The doubly conditional  reflected BSDE \eqref{nonlinearyz} degenerates to the BSDE with conditional reflection studied in \cite{HHL}. 

    \noindent (ii) For the case that $\mathbb{G}=\mathbb{F}$, the doubly conditional  reflected BSDE \eqref{nonlinearyz} turns into the doubly reflected BSDE (see \cite{CK,HH,HL,PX}). For the case that $\mathbb{G}$ is the deterministic scenario, the doubly conditional reflected BSDE \eqref{nonlinearyz} becomes the doubly mean reflected BSDE (see \cite{FS,L}).
\end{remark}

Based on the Skorokhod problem in a time-dependent interval (see \cite{BKR,Slaby}), we have the following pointwise representation for the bounded variation term $K$, which will be helpful to establish some a priori estimates for solutions to doubly conditional reflected BSDEs.
\begin{proposition}\label{repK}
    Under Assumptions (H1)-(H3), suppose that $(Y,Z,K)\in \mathcal{S}^2\times \mathcal{H}^2\times \mathcal{BV}^2_{\mathbb{G}}$ is a solution to the doubly conditional reflected BSDE \eqref{nonlinearyz}. Then, for any $t\in[0,T]$ and $\omega\in \Omega$, we have
    \begin{equation*}\label{representationforK}
        \begin{split}
            K_T(\omega)-K_t(\omega)=-\max\Big\{&0\wedge \inf_{v\in[t,T]}(a^\omega+x^\omega_T-x_v^\omega-l_v^\omega),\\
            &\sup_{r\in[t,T]}\left[(a^\omega+x^\omega_T-x_r^\omega-u_r^\omega)\wedge \inf_{v\in[t,r]}(a^\omega+x^\omega_T-x_v^\omega-l_v^\omega)\right]\Big\},
        \end{split}
    \end{equation*}
    where $a^\omega=\E[\xi|\mathcal{G}_T](\omega)$ and 
    \begin{align*}
    &l_t^\omega=\E[L_t|\mathcal{G}_t](\omega), \ u_t^\omega=\E[U_t|\mathcal{G}_t](\omega),\\
     &   x_t^\omega=\E\left[\left(\int_0^t f(s,Y_s,Z_s)ds-\int_0^t Z_s dB_s\right)\Big|\mathcal{G}_t\right](\omega).
    \end{align*}
\end{proposition}

\begin{proof}
    Rewrite \eqref{nonlinearyz} in the forward form and taking conditional expecations yield that 
    \begin{align*}
        \E[Y_t|\mathcal{G}_t]=\left(Y_0-\E\left[\left(\int_0^t f(s,Y_s,Z_s)ds-\int_0^t Z_s dB_s\right)\Big|\mathcal{G}_t\right]\right)-K_t.
    \end{align*}
   Consequently, we have
   \begin{align*}
       \E[Y_t|\mathcal{G}_t]-\E[\xi|\mathcal{G}_T]=&\Bigg( \E\left[\int_0^T f(s,Y_s,Z_s)ds\Big|\mathcal{G}_T\right]- \E\left[\int_0^t f(s,Y_s,Z_s)ds\Big|\mathcal{G}_t\right]\\
      & + \E\left[\int_0^t Z_sdB_s\Big|\mathcal{G}_t\right]- \E\left[\int_0^T Z_sdB_s\Big|\mathcal{G}_T\right]\Bigg)+K_T-K_t.
   \end{align*}
   For each fixed $\omega$, set $\tilde{y}^\omega_t=\E[Y_{T-t}|\mathcal{G}_{T-t}](\omega)$, $\tilde{l}^\omega_t=\E[L_{T-t}|\mathcal{G}_{T-t}](\omega)$, $\tilde{u}^\omega_t=\E[U_{T-t}|\mathcal{G}_{T-t}](\omega)$, $\tilde{k}^\omega_t=(K_T-K_{T-t})(\omega)$ and 
   \begin{align*}
       \tilde{x}^\omega_t=&\Bigg( \E\left[\int_0^T f(s,Y_s,Z_s)ds\Big|\mathcal{G}_T\right]- \E\left[\int_0^{T-t} f(s,Y_s,Z_s)ds\Big|\mathcal{G}_{T-t}\right]\\
      & + \E\left[\int_0^{T-t} Z_sdB_s\Big|\mathcal{G}_{T-t}\right]- \E\left[\int_0^T Z_sdB_s\Big|\mathcal{G}_T\right]+\E[\xi|\mathcal{G}_T|]\Bigg)(\omega)\\
      =&a^\omega+x^\omega_T-x^\omega_{T-t}.
   \end{align*}
   Then, $(\tilde{y}^\omega,\tilde{k}^\omega)$ is the solution to the Skorokhod problem on $[\tilde{l}^\omega,\tilde{u}^\omega]$ for $\tilde{x}^\omega$ (see Definition 2.1 in \cite{BKR} or Definition 1.2 in \cite{Slaby}). That is, we have
   \begin{displaymath}
        \begin{cases}
        \tilde{y}^\omega_t=\tilde{x}^\omega_t+\tilde{k}^\omega_t,\\
        \tilde{l}_t^\omega\leq \tilde{y}_t^\omega\leq \tilde{u}_t^\omega, \ \tilde{k}_t^\omega=\tilde{k}_t^{\omega,+}-\tilde{k}_t^{\omega,-},, \ \tilde{k}_t^{\omega,+},\tilde{k}_t^{\omega,-}\in I[0,T],\\
         \int_0^T (\tilde{y}_t^\omega-\tilde{l}_t^\omega)d\tilde{k}^{\omega,+}_t=\int_0^T (\tilde{u}_t^\omega-\tilde{y}_t^\omega)d\tilde{k}^{\omega,-}_t=0.
        \end{cases}
    \end{displaymath}    
    By Theorem 2.6 in \cite{BKR} or Theorem 2.1 in \cite{Slaby}, we have
    \begin{align*}
        \tilde{k}^\omega_t=-\max\left\{\left[(\tilde{x}^\omega_0-\tilde{u}^\omega_0)^+\wedge \inf_{r\in[0,t]}(\tilde{x}^\omega_r-\tilde{l}^\omega_r)\right],\sup_{s\in[0,t]}\left[(\tilde{x}^\omega_s-\tilde{u}^\omega_s)\wedge \inf_{r\in[s,t]}(\tilde{x}^\omega_r-\tilde{l}^\omega_r)\right]\right\}.
    \end{align*}
    The proof is complete.
\end{proof}

Recall that if $(y^i,k^i)$ are the solutions to the Skorokhod problem on $[l^i,u^i]$ for $x^i$, $i=1,2$, then, by Proposition 4.1 in \cite{Slaby}, there exists a constant $C>0$, such that 
\begin{align}\label{diffk}
    \sup_{t\in[0,T]}|k^1_t-k^2_t|\leq C\left\{\sup_{t\in[0,T]}|x^1_t-x^2_t|+\sup_{t\in[0,T]}\max(|l^1_t-l^2_t|,|u^1_t-u^2_t|)\right\}.
\end{align}
We first introduce some a priori estimates for the solutions to the doubly conditional  reflected BSDEs.

\begin{proposition}\label{thm2.3}
    Let $(Y^i,Z^i,K^i)\in \mathcal{S}^2\times\mathcal{H}^2\times\mathcal{BV}_{\mathbb{G}}^2$ be the solution to the doubly conditional reflected BSDE with parameters $(\xi^i,f^i,L,U)$ satisfying Assumptions (H1)-(H3), $i=1,2$. Then, there exists a constant $C$ depending on $\lambda,T$, such that 
    \begin{align*}
        &\E\left[\sup_{t\in[0,T]}|Y^1_t-Y^2_t|^2+\int_0^T |Z^1_t-Z^2_t|^2dt+\sup_{t\in[0,T]}|K^1_t-K^2_t|^2\right]\\
        &\leq C\E\left[|\hat{\xi}|^2+\int_0^T|f^1(s,Y^2_s,Z^2_s)-f^2(s,Y^2_s,Z^2_s)|^2ds\right].
    \end{align*}
\end{proposition}

\begin{proof}
    For simplicity, we denote $\hat{A}=A^1-A^2$ for $A=Y,Z,K,\xi$ and
    \begin{align*}
        \hat{f}(s)&=f^1(s,Y^1_s,Z^1_s)-f^2(s,Y^2_s,Z^2_s), \\
        \hat{f}^1(s)&=f^1(s,Y^1_s,Z^1_s)-f^1(s,Y^2_s,Z^2_s),\\
        \hat{f}(s,Y^2_s,Z^2_s)&=f^1(s,Y^2_s,Z^2_s)-f^2(s,Y^2_s,Z^2_s).
    \end{align*}
    For any $\beta>0$, applying It\^{o}'s formula to $e^{\beta t}|\hat{Y}_t|^2$, we obtain that 
    \begin{equation}\label{e2.5}
    \begin{split}
        &e^{\beta t}|\hat{Y}_t|^2+\int_t^T e^{\beta s}(\beta |\hat{Y}_s|^2+|\hat{Z}_s|^2)ds\\
        =&e^{\beta T}|\hat{\xi}|^2+2\int_t^T e^{\beta s}\hat{f}(s)\hat{Y}_s ds+2\int_t^T e^{\beta s}\hat{Y}_sd\hat{K}_s-2\int_t^T e^{\beta s}\hat{Y}_s\hat{Z}_sdB_s\\
        \leq &e^{\beta T}|\hat{\xi}|^2+\int_t^T e^{\beta s}|\hat{f}(s,Y^2_s,Z^2_s)|^2 ds+(1+2\lambda+8\lambda^2)\int_t^T e^{\beta s}|\hat{Y}_s|^2\\
        &+\frac{1}{8}\int_t^T e^{\beta s}|\hat{Z}_s|^2ds+2\int_t^T e^{\beta s}\hat{Y}_sd\hat{K}_s-2\int_t^T e^{\beta s}\hat{Y}_s\hat{Z}_sdB_s.
    \end{split}
    \end{equation}
    Simple calculation yields that 
    \begin{equation}\label{e2.5'}\begin{split}
        &\E\left[\int_t^T e^{\beta s}\hat{Y}_sd\hat{K}_s\Big|\mathcal{G}_t\right]=\E\left[\int_t^T e^{\beta s}\E[\hat{Y}_s|\mathcal{G}_s]d\hat{K}_s\Big|\mathcal{G}_t\right]\\
        =&\E\left[\int_t^T e^{\beta s}\E[Y_s^1-L_s|\mathcal{G}_s]dK^{1,+}_s\Big|\mathcal{G}_t\right]+\E\left[\int_t^T e^{\beta s}\E[L_s-Y_s^2|\mathcal{G}_s]dK^{1,+}_s\Big|\mathcal{G}_t\right]\\
        &-\E\left[\int_t^T e^{\beta s}\E[Y_s^1-U_s|\mathcal{G}_s]dK^{1,-}_s\Big|\mathcal{G}_t\right]-\E\left[\int_t^T e^{\beta s}\E[U_s-Y_s^2|\mathcal{G}_s]dK^{1,-}_s\Big|\mathcal{G}_t\right]\\
        &-\E\left[\int_t^T e^{\beta s}\E[Y_s^1-L_s|\mathcal{G}_s]dK^{2,+}_s\Big|\mathcal{G}_t\right]-\E\left[\int_t^T e^{\beta s}\E[L_s-Y_s^2|\mathcal{G}_s]dK^{2,+}_s\Big|\mathcal{G}_t\right]\\
        &+\E\left[\int_t^T e^{\beta s}\E[Y_s^1-U_s|\mathcal{G}_s]dK^{2,-}_s\Big|\mathcal{G}_t\right]+\E\left[\int_t^T e^{\beta s}\E[U_s-Y_s^2|\mathcal{G}_s]dK^{2,-}_s\Big|\mathcal{G}_t\right]\leq 0.
    \end{split}
    \end{equation}
    Choosing $\beta=2+2\lambda+8\lambda^2$, it follows from \eqref{e2.5} and \eqref{e2.5'} that 
    \begin{equation}\label{e2.4}
        \E\left[|\hat{Y}_t|^2+\int_t^T (|\hat{Y}_s|^2+|\hat{Z}_s|^2)ds\Big |\mathcal{G}_t\right]\leq C\E\left[|\hat{\xi}|^2+\int_t^T |\hat{f}(s,Y^2_s,Z^2_s)|^2ds \Big|\mathcal{G}_t\right].
    \end{equation}

    By the proof of Proposition \ref{repK} and \eqref{diffk}, we have
    \begin{equation*}
    \begin{split}
        \sup_{t\in[0,T]}|\hat{K}_t|\leq C|\E[\hat{\xi}|\mathcal{G}_T]|+C\sup_{t\in[0,T]}\Bigg|&\E\left[ \left(\int_0^T \hat{f}(s)ds-\int_0^T\hat{Z}_s dB_s\right)\Big|\mathcal{G}_T \right]\\ 
        &-\E\left[ \left(\int_0^t \hat{f}(s)ds-\int_0^t\hat{Z}_s dB_s\right)\Big|\mathcal{G}_t \right]\Bigg|.
    \end{split}\end{equation*}
    It follows from the Doob martingale inequality, the Burkholder-Davis-Gundy inequality, the H\"{o}lder inequality and \eqref{e2.4} that 
    \begin{align}\label{e2.11}
        \E\left[\sup_{t\in[0,T]}|\hat{K}_t|^2\right]\leq C \left\{\E[|\hat{\xi}|^2]+\E\left[\int_0^T|\hat{f}(s,Y^2_s,Z^2_s)|^2ds\right]\right\}.
    \end{align}
    Note that 
    \begin{align*}
        \hat{Y}_t=\E[\hat{\xi}|\mathcal{F}_t]+\E\left[\int_t^T \hat{f}^1(s)+\hat{f}(s,Y^2_s,Z^2_s)ds\Big|\mathcal{F}_t\right]+\E[\hat{K}_T-\hat{K}_t|\mathcal{F}_t].
    \end{align*}
    Applying \eqref{e2.4} and \eqref{e2.11}, we obtain that 
    \begin{align}\label{e2.11'}
        \E\left[\sup_{t\in[0,T]}|\hat{Y}_t|^2\right]\leq C\E\left[|\hat{\xi}|^2+\int_0^T|\hat{f}(s,Y^2_s,Z^2_s)|^2ds\right].
    \end{align}
    Combining \eqref{e2.4}, \eqref{e2.11} and \eqref{e2.11'}, we obtain the desired result.
\end{proof}

Proposition \ref{thm2.3} implies the following uniqueness result directly. 

\begin{theorem}\label{uniqueness}
    Let the parameters $(\xi,f,L,U)$ satisfy Assumptions (H1)-(H3). Then, the doubly conditional reflected BSDE \eqref{nonlinearyz} has at most a solution $(Y,Z,K)\in \mathcal{S}^2\times\mathcal{H}^2\times \mathcal{BV}_{\mathbb{G}}^2$.
\end{theorem}

In order to obtain the existence result for the doubly conditional  reflected BSDE, we need to propose the following condition for the obstacles, known as the Mokobodski condition (see \cite{CK,HL,PX}). 
\begin{itemize}
    \item[(H4)] There exist two nonnegative $\mathbb{F}$-supermartingale $H^1,H^2\in \mathcal{S}^2$ such that for any $t\in[0,T]$,
    \begin{align*}
        L_t\leq H^1_t-H^2_t\leq U_t.
    \end{align*}
\end{itemize}

\begin{remark}
    Let $L'_t=\E[L_t|\mathcal{G}_t]$, $U'_t=\E[U_t|\mathcal{G}_t]$. Set $H^{1'}_t=\E[H^1_t|\mathcal{G}_t]$, $H^{2'}_t=\E[H^2_t|\mathcal{G}_t]$. Simple calculation yields that for $i=1,2$, $0\leq s\leq t\leq T$, 
    \begin{align*}
        \E[H^{i'}_t|\mathcal{G}_s]=\E[H^i_t|\mathcal{G}_s]=\E[\E[H^i_t|\mathcal{F}_s]|\mathcal{G}_s]\leq \E[H^i_s|\mathcal{G}_s]=H^{i'}_s,
    \end{align*}
    which implies that $H^{i'}$, $i=1,2$ are nonnegative $\mathbb{G}$-supermartingale. Then, under (H4), the Mokobodski condition holds for $L',U'$, that is, there exist two nonnegative $\mathbb{G}$-supermartingale $H^{1'},H^{2'}\in \mathcal{S}^2$ such that for any $t\in[0,T]$,
    \begin{align*}\label{constraints}
        L'_t\leq H^{1'}_t-H^{2'}_t\leq U'_t.
    \end{align*}
\end{remark}

We first construct the solution to the doubly conditional reflected BSDE whose driver $f$ does not depend on $(y,z)$. Recalling the single reflected case studied in \cite{HHL}, the construction for the solution is based on the Snell envelope approach, i.e., the optimal stopping problem with partial information. For the doubly reflected case, a natural candidate for the construction is the Dynkin game with partial information. 

\begin{proposition}\label{prop2.6}
   Let $f\in\mathcal{H}^2$. Under Assumptions (H2)-(H4), there exists a unique solution $(Y,Z,K)\in \mathcal{S}^2\times \mathcal{H}^2\times \mathcal{BV}^2_{\mathbb{G}}$ to the following doubly conditional  reflected BSDE
   \begin{equation*}\label{nonlinearnoyz}
\begin{cases}
Y_t=\xi+\int_t^T f(s)ds-\int_t^T Z_s dB_s+K_T-K_t, \\
\E[L_t|\mathcal{G}_t]\leq \E[Y_t|\mathcal{G}_t]\leq \E[U_t|\mathcal{G}_t], \\
K_t=K^+_t-K^-_t, \ K^+,K^-\in \mathcal{A}^2_{\mathbb{G}},\\
\int_0^T \E[Y_t-L_t|\mathcal{G}_t]dK_t^+=\int_0^T \E[U_t-L_t|\mathcal{G}_t]dK^-_t=0.
\end{cases}
\end{equation*}
\end{proposition}

\begin{proof}
     For any $t\in[0,T]$, let $\mathcal{T}_{t,T}^{\mathbb{G}}$ be the collection of $\mathbb{G}$-adapted stopping times taking values in $[t,T]$. For simplicity, we omit the superscript $\mathbb{G}$ in the proof. For any $\sigma,\tau\in \mathcal{T}_{t,T}$, consider the payoff
     \begin{align*}\label{Rsigmatau}
         R_t(\sigma,\tau):=\int_t^{\sigma\wedge \tau} f(u)du+\xi I_{\{\sigma\wedge \tau=T\}}+L_\tau I_{\{\tau<T,\tau\leq \sigma\}}+U_\sigma I_{\{\sigma<\tau\}}
     \end{align*}
     and the upper and lower value of the associated Dynkin game
     \begin{equation}\label{dynkin1}\begin{split}
         &\bar{V}_t:=\essinf_{\sigma\in\mathcal{T}_{t,T}}\esssup_{\tau\in\mathcal{T}_{t,T}}\E[R_t(\sigma,\tau)|\mathcal{G}_t],\\
         &\underline{V}_t:=\esssup_{\tau\in\mathcal{T}_{t,T}}\essinf_{\sigma\in\mathcal{T}_{t,T}}\E[R_t(\sigma,\tau)|\mathcal{G}_t].
     \end{split}\end{equation}
     Set $\widetilde{R}_t(\sigma,\tau)=\widetilde{L}_\tau I_{\{\tau\leq \sigma\}}+\widetilde{U}_\sigma I_{\{\sigma<\tau\}}$, where
     \begin{align*}
         &\widetilde{L}_t=\E\left[\int_0^t f(u)du+L_t I_{\{t<T\}}+\xi I_{\{t=T\}}\Big|\mathcal{G}_t\right],\\
         &\widetilde{U}_t=\E\left[\int_0^t f(u)du+U_tI_{\{t<T\}}+\xi I_{\{t=T\}}\Big|\mathcal{G}_t\right].
     \end{align*}
     It is easy to check that 
     \begin{equation}\label{dynkin2}\begin{split}
         &\bar{V}'_t:=\bar{V}_t+\E\left[\int_0^t f(u)du\Big|\mathcal{G}_t\right]=\essinf_{\sigma\in\mathcal{T}_{t,T}}\esssup_{\tau\in\mathcal{T}_{t,T}}\E[\widetilde{R}_t(\sigma,\tau)|\mathcal{G}_t],\\
         &\underline{V}'_t:=\underline{V}_t+\E\left[\int_0^t f(u)du\Big|\mathcal{G}_t\right]=\esssup_{\tau\in\mathcal{T}_{t,T}}\essinf_{\sigma\in\mathcal{T}_{t,T}}\E[\widetilde{R}_t(\sigma,\tau)|\mathcal{G}_t].
     \end{split}\end{equation}
     By Theorem \ref{theorem2.3}, there exist two $\mathbb{G}$-supermartingale $J^1,J^2$ such that
     \begin{align*}
         J^1_t=\esssup_{\tau\in \mathcal{T}_{t,T}}\E[J^2_\tau+\widetilde{L}_\tau|\mathcal{G}_t], \ J^2_t=\esssup_{\tau\in \mathcal{T}_{t,T}}\E[J^1_\tau-\widetilde{U}_\tau|\mathcal{G}_t].
     \end{align*}
     By Proposition \ref{proposition5.2} and Remark \ref{r2.11}, we have $J^i\in \mathcal{S}^2$, $i=1,2$. Let $K^+,K^-\in\mathcal{A}_{\mathbb{G}}^2$ be the nondecreasing term in the Doob-Meyer decomposition of $J^1,J^2$, respectively. By Proposition B.11 in \cite{KQ}, we have 
     \begin{equation}\label{flatoff}
         \begin{split}
             \int_0^T (J^1_t-(J^2_t+\widetilde{L}_t))dK^+_t=\int_0^T(J^2_t-(J^1_t-\widetilde{U}_t))dK^-_t=0.
         \end{split}
     \end{equation}
     By Proposition \ref{proposition5.2} and Theorem \ref{theorem3.6}, the value of the Dynkin game \eqref{dynkin2} exists, denoted by $Y'$. Besides, we have
     \begin{align}\label{reprentationY'}
         Y'_t=J^1_t-J^2_t.
     \end{align}
     Proposition \ref{proposition2.5} implies that for any $t\in[0,T]$,
     \begin{align}\label{tildeLYtileU}
         \widetilde{L}_t\leq Y'_t\leq \widetilde{U}_t.
     \end{align}
     Set $K_t=K^+_t-K^-_t$. It follows from \eqref{reprentationY'} and the definition of $K^+,K^-$ that 
     \begin{align}\label{yk}
         \E[Y'_t-Y'_T|\mathcal{G}_t]=\E[K_T-K_t|\mathcal{G}_t].
     \end{align}
     Since the value of the Dynkin game \eqref{dynkin2} exists, the Dynkin game \eqref{dynkin1} is fair, i.e., we have 
     \begin{align*}
         \bar{V}_t=\underline{V}_t=:\bar{Y}_t.
     \end{align*}
     Noting that for any $t\in[0,T]$,
     \begin{align}\label{Y'Y}
         Y'_t=\bar{Y}_t+\E\left[\int_0^t f(u)du\Big|\mathcal{G}_t\right],
     \end{align}
     which together with \eqref{tildeLYtileU} yields that for any $t\in[0,T]$
     \begin{align}\label{LYU}
         \E[L_t|\mathcal{G}_t]\leq \bar{Y}_t\leq \E[U_t|\mathcal{G}_t].
     \end{align}
     Furthermore, we have
     \begin{align}\label{barYY'}
         \bar{Y}_t=\bar{Y}_T+\E\left[\int_0^T f(u)du\Big|\mathcal{G}_T\right]-\E\left[\int_0^t f(u)du\Big|\mathcal{G}_t\right]+Y'_t-Y'_T.
     \end{align}
     Since $\bar{Y}$ is $\mathbb{G}$-adpated and $\bar{Y}_T=\E[\xi|\mathcal{G}_T]$, taking conditional expectations on both sides of \eqref{barYY'} and applying \eqref{yk} yield that 
     \begin{align*}
         \bar{Y}_t=\E\left[\xi+\int_t^T f(u)du+K_T-K_t\Big|\mathcal{G}_t\right].
     \end{align*}

     Let $(Y,Z)\in \mathcal{S}^2\times \mathcal{H}^2$ be the solution to the following BSDE
     \begin{align*}
         Y_t=\xi+\int_t^T f(u)du+K_T-K_t-\int_t^T Z_udB_u.
     \end{align*}
     We claim that $(Y,Z,K)$ is the solution to conditional doubly reflected BSDE. In fact, it is easy to check that $\E[Y_t|\mathcal{G}_t]=\bar{Y}_t$. Then, \eqref{LYU} is equivalent to  the desired constraints
     \begin{align*}
         \E[L_t|\mathcal{G}_t]\leq \E[Y_t|\mathcal{G}_t]\leq \E[U_t|\mathcal{G}_t], \textrm{ for all }t\in[0,T].
     \end{align*}
     It remains to prove the Skorokhod conditions hold. By Eqs. \eqref{flatoff}, \eqref{reprentationY'} and \eqref{Y'Y}, we have
     \begin{align*}
         &\int_0^T (\E[Y_t|\mathcal{G}_t]-\E[L_t|\mathcal{G}_t])dK^+_t=\int_0^T (Y'_t-\widetilde{L}_t)dK^+_t=0,\\
          &\int_0^T (\E[L_t|\mathcal{G}_t]-\E[Y_t|\mathcal{G}_t])dK^-_t=\int_0^T (\widetilde{U}_t-Y'_t)dK^-_t=0.
     \end{align*}
     The proof is complete.
\end{proof}

Now, we are ready to introduce the main result of this section.

\begin{theorem}\label{thm2.7}
    Suppose that the parameters $(\xi,f,L,U)$ satisfy (H1)-(H4). Then, the doubly conditional reflected BSDE has a unique solution $(Y,Z,K)\in \mathcal{S}^2\times \mathcal{H}^2\times \mathcal{BV}^2_{\mathbb{G}}$.
\end{theorem}

\begin{proof}
    For any given $(U,V)\in \mathcal{S}^2\times \mathcal{H}^2$, by Proposition \ref{prop2.6}, there exists a unique solution $(Y,Z,K)\in \mathcal{S}^2\times \mathcal{H}^2\times \mathcal{BV}^2_{\mathbb{G}}$  to the following equation
    \begin{displaymath}
\begin{cases}
Y_t=\xi+\int_t^T f(s,U_s,V_s)ds-\int_t^T Z_s dB_s+K_T-K_t, \\
\E[L_t|\mathcal{G}_t]\leq \E[Y_t|\mathcal{G}_t]\leq \E[U_t|\mathcal{G}_t], \\
K_t=K^+_t-K^-_t, \ K^+,K^-\in \mathcal{A}^2_{\mathbb{G}},\\
\int_0^T \E[Y_t-L_t|\mathcal{G}_t]dK_t^+=\int_0^T \E[U_t-L_t|\mathcal{G}_t]dK^-_t=0.
\end{cases}
\end{displaymath}
We define a mapping $\Phi:\mathcal{S}^2\times\mathcal{H}^2\rightarrow\mathcal{S}^2\times\mathcal{H}^2$ as follows:
\begin{align*}
    \Phi(U,V)=(Y,Z).
\end{align*}
We first prove that $\Phi$ is a contraction mapping from $\mathcal{H}^2\times\mathcal{H}^2$ to $\mathcal{H}^2\times\mathcal{H}^2$ with the norm
\begin{align*}
    \|(Y,Z)\|_\beta^2=\E\left[\int_0^T e^{\beta t}(|Y_t|^2+|Z_t|^2)dt\right],
\end{align*}
where $\beta$ is a positive constant to be determined later. In fact, given $(U^i,V^i)\in \mathcal{S}^2\times \mathcal{H}^2$, $i=1,2$, let $(Y^i,Z^i)=\Phi(U^i,V^i)$ and we denote $\hat{A}=A^1-A^2$ for $A=Y,Z,K$ and 
\begin{align*}
    \hat{f}(s)&=f(s,U^1_s,V^1_s)-f(s,U^2_s,V^2_s).
\end{align*}
Applying It\^{o}'s formula to $e^{\beta t}|\hat{Y}_t|^2$ and taking conditional expectations, we obtain that 
\begin{align*}
        &\E[e^{\beta t}|\hat{Y}_t|^2|\mathcal{G}_t]+\E\left[\int_t^T e^{\beta s}(\beta |\hat{Y}_s|^2+|\hat{Z}_s|^2)ds\Big|\mathcal{G}_t\right]\\
        =&\E\left[2\int_t^T e^{\beta s}\hat{f}(s)\hat{Y}_s ds\Big|\mathcal{G}_t\right]+\E\left[2\int_t^T e^{\beta s}\hat{Y}_sd\hat{K}_s\Big|\mathcal{G}_t\right]\\
        \leq &4\lambda^2\E\left[\int_t^T e^{\beta s}|\hat{Y}_s|^2 ds\Big|\mathcal{G}_t\right]+\frac{1}{2}\E\left[\int_t^T e^{\beta s}( |\hat{U}_s|^2+|\hat{V}_s|^2)ds\Big|\mathcal{G}_t\right],
\end{align*}
where we have used Eq. \eqref{e2.5'} in the last inequality. Choosing $\beta=4\lambda^2+1$, we obtain that 
\begin{equation}\label{e2.20}
    \E[e^{\beta t}|\hat{Y}_t|^2|\mathcal{G}_t]+\E\left[\int_t^T e^{\beta s}( |\hat{Y}_s|^2+|\hat{Z}_s|^2)ds\Big|\mathcal{G}_t\right]\leq \frac{1}{2}\E\left[\int_t^T e^{\beta s}( |\hat{U}_s|^2+|\hat{V}_s|^2)ds\Big|\mathcal{G}_t\right],
\end{equation}
which implies that $\Phi$ is a contraction mapping on $\mathcal{H}^2\times\mathcal{H}^2$. 

By Proposition \ref{repK} and Eq. \eqref{diffk}, we have
\begin{equation}\label{e2.21}
    \begin{split}
        \sup_{t\in[0,T]}|\hat{K}_t|\leq& C\sup_{t\in[0,T]}\left|\E\left[\int_0^T \hat{f}(s)ds-\int_0^T \hat{Z}_sdB_s\Big|\mathcal{G}_T\right]-\E\left[\int_0^t \hat{f}(s)ds-\int_0^t \hat{Z}_sdB_s\Big|\mathcal{G}_t\right]\right|\\
        \leq & C\E\left[\int_0^T ( |\hat{U}_s|+|\hat{V}_s|)ds\Big|\mathcal{G}_T\right]+C \sup_{t\in[0,T]}\E\left[\int_0^T ( |\hat{U}_s|+|\hat{V}_s|)ds\Big|\mathcal{G}_t\right]\\
        &+C\sup_{t\in[0,T]}\E\left[\left|\int_0^t \hat{Z}_sdB_s\right|\Big|\mathcal{G}_t\right]+C\E\left[\left|\int_0^T \hat{Z}_sdB_s\right|\Big|\mathcal{G}_T\right].
    \end{split}
\end{equation}
Note that 
\begin{align}\label{e2.20'}
    \hat{Y}_t=\E\left[\int_t^T \hat{f}(s)ds\Big|\mathcal{F}_t\right]+\E[\hat{K}_T-\hat{K}_t|\mathcal{F}_t].
\end{align}
Combining Eqs. \eqref{e2.20}-\eqref{e2.20'} and applying Doob's martingale inequality yield that 
\begin{align*}
    \E\left[\sup_{t\in[0,T]}|\hat{Y}_t|^2\right]\leq C\E\left[\int_0^T ( |\hat{U}_s|^2+|\hat{V}_s|^2)ds\right],
\end{align*}
which implies that $\Phi$ is continuous from $\mathcal{S}^2\times \mathcal{H}^2$ to itself. Together with \eqref{e2.20}, $\Phi$ has a fixed point $(Y,Z)\in \mathcal{S}^2\times \mathcal{H}^2$. $K$ can be constructed as in the proof of Proposition \ref{repK}. The proof is complete.
\end{proof}

\begin{remark}
  (i)  Compared with the doubly mean reflected BSDE studied in \cite{FS,L}, due to the fact that the bounded variation term $K$ is no longer deterministic, the proof for the wellposedness is more complicated. Furthermore, recall that the constraints in the mean reflected case may be nonlinear. More precisely, the constraints take the following form $\E[h(t,Y_t)]\in[l_t,u_t]$ (resp., $\E[L(t,Y_t)]\leq 0\leq \E[R(t,Y_t)]$) in \cite{FS} (resp., \cite{L}). A natural extension for the conditional reflected case is that 
    \begin{align*}
        \E[L(t,Y_t)|\mathcal{G}_t]\leq 0\leq \E[R(t,Y_t)|\mathcal{G}_t].
    \end{align*}
    Unfortunately, the method used in the linear reflected case is not valid. We will leave it for the future research.

    \noindent (ii) Another effective method to construct the solution to the classical reflected BSDEs is the approximation via penalization, which has also been applied to the mean reflected case with linear loss function. For the conditional reflected case, the proof by the penalization method is put into the Appendix. 
\end{remark}

\section{Properties of solutions to doubly conditional  reflected BSDEs}

Recall that for the classical doubly reflected BSDE, the solution corresponds to the value function of a certain Dynkin game (see \cite{CK,HH,HL}) and for the doubly mean reflected BSDE, the expectation of the solution coincides with some optimization problem (see \cite{FS,L}). A natural question is that if the solution to the doubly conditional reflected BSDE has a similar representation. We give an affirmative answer below. More precisely, let $(Y,Z,K)$ be the solution to the conditional doubly reflected BSDE \eqref{nonlinearyz} and Let $(Y',Z',K')$ be the solution to the doubly conditional  reflected BSDE with terminal value $\xi$, obstacles $L,U$ and constant driver $\{f(s,Y_s,Z_s)\}_{s\in[0,T]}$. Due to the uniqueness result (see Theorem \ref{uniqueness}), we have $Y\equiv Y'$. According to the proof of Proposition \ref{prop2.6}, we have the following result.

\begin{corollary}\label{cor3.1}
    Let $(Y,Z,K)$ be the solution to \eqref{nonlinearyz}. Then, for any $t\in[0,T]$, we have
    \begin{align*}
\E[Y_t|\mathcal{G}_t]=\essinf_{\sigma\in\mathcal{T}^{\mathbb{G}}_{t,T}}\esssup_{\tau\in\mathcal{T}^{\mathbb{G}}_{t,T}}\E\left[\widetilde{R}_t(\sigma,\tau)|\mathcal{G}_t\right]=\esssup_{\tau\in\mathcal{T}^{\mathbb{G}}_{t,T}}\essinf_{\sigma\in\mathcal{T}^{\mathbb{G}}_{t,T}}\E\left[\widetilde{R}_t(\sigma,\tau)|\mathcal{G}_t\right],
    \end{align*}
    where 
     \begin{align*}
         \widetilde{R}_t(\sigma,\tau):=\int_t^{\sigma\wedge \tau} f(u,Y_u,Z_u)du+\xi I_{\{\sigma\wedge \tau=T\}}+L_\tau I_{\{\tau<T,\tau\leq \sigma\}}+U_\sigma I_{\{\sigma<\tau\}}.
     \end{align*}
\end{corollary}

\begin{remark}
    If $\mathbb{G}=\mathbb{F}$, Corollary \ref{cor3.1} degenerates into Theorem 4.1 in \cite{CK} and Proposition 2.2.1 in \cite{HH}. If $\mathbb{G}$ is chosen to be a deterministic scenario, Corollary \ref{cor3.1} degenerates into  Theorem 3.6 in \cite{FS} and Theorem 4.7 in \cite{L} with the loss functions in \cite{FS,L} being linear. 
\end{remark}

Corollary \ref{cor3.1} indicates that the conditional expectation of the solution to the doubly conditional reflected BSDE coincides with the value function of a Dynkin game with partial information. However, the payoff process $\widetilde{R}$ also depends on the solution $Y,Z$. In the sequel, we aim to find another representation of the Dynkin game such that its corresponding payoff process is independent of the solution to the doubly conditional  reflected BSDE. This representation needs some additional assumptions for the driver $f$.

\begin{theorem}\label{thm3.2}
    Let $(Y,Z,K)$ be the unique solution to the following doubly  conditional reflected BSDE
    \begin{equation*}\label{linearyz}
\begin{cases}
Y_t=\xi+\int_t^T (a_sY_s+b_sZ_s+c_s)ds-\int_t^T Z_s dB_s+K_T-K_t, \\
\E[L_t|\mathcal{G}_t]\leq \E[Y_t|\mathcal{G}_t]\leq \E[U_t|\mathcal{G}_t], \\
K_t=K^+_t-K^-_t, \ K^+,K^-\in \mathcal{A}^2_{\mathbb{G}},\\
\int_0^T \E[Y_t-L_t|\mathcal{G}_t]dK_t^+=\int_0^T \E[U_t-L_t|\mathcal{G}_t]dK^-_t=0.
\end{cases}
\end{equation*}
Let $\Gamma$ be the unique solution to the following linear SDE
\begin{equation}\label{e3.2}
    d\Gamma_s=a_s\Gamma_sds+b_s\Gamma_sdB_s, \ \Gamma_0=1.
\end{equation}
Suppose that the process $\Gamma$ is $\mathbb{G}$-adapted. For each fixed $t\in[0,T]$, for any given $\tau,\sigma\in\mathcal{T}^{\mathbb{G}}_{t,T}$, let $(y^{\tau,\sigma},z^{\tau,\sigma})$ be the unique solution to the following BSDE
\begin{equation}\label{e3.4}
    y^{\tau,\sigma}_s=(\xi I_{\{\sigma\wedge \tau=T\}}+L_\tau I_{\{\tau<T,\tau\leq \sigma\}}+U_\sigma I_{\{\sigma<\tau\}})+\int_s^{\tau\wedge\sigma}(a_r y^{\tau,\sigma}_r+b_r z^{\tau,\sigma}_r+c_r)dr-\int_s^{\tau\wedge\sigma}z^{\tau,\sigma}_rdB_r.
\end{equation}
Then, for any $t\in[0,T]$, we have
\begin{equation}\label{e3.3}
\E[Y_t|\mathcal{G}_t]=\essinf_{\sigma\in\mathcal{T}^{\mathbb{G}}_{t,T}}\esssup_{\tau\in\mathcal{T}^{\mathbb{G}}_{t,T}}\E\left[y_t^{\sigma,\tau}|\mathcal{G}_t\right]=\esssup_{\tau\in\mathcal{T}^{\mathbb{G}}_{t,T}}\essinf_{\sigma\in\mathcal{T}^{\mathbb{G}}_{t,T}}\E\left[y_t^{\sigma,\tau}|\mathcal{G}_t\right].
\end{equation}
Moreover, the saddle point $(\tau^*,\sigma^*)$ is given by
\begin{equation}\label{e3.5}
 \begin{split}
     &\tau^*_t=\inf\{s\in[t,T]:\E[Y_s-L_s|\mathcal{G}_s]=0\}\wedge T, \\
     &\sigma^*_t=\inf\{s\in[t,T]:\E[Y_s-U_s|\mathcal{G}_s]=0\}\wedge T.
 \end{split}   
\end{equation}
\end{theorem}

\begin{proof}
    In order to prove \eqref{e3.3} and \eqref{e3.5}, it suffices to show that for any $\tau,\sigma\in\mathcal{T}^{\mathbb{G}}_{t,T}$, we have
    \begin{equation*}
        \E\left[y^{\tau\wedge\sigma^*_t}_t\big|\mathcal{G}_t\right]\leq \E[Y_t|\mathcal{G}_t]\leq \E\left[y^{\tau^*_t\wedge\sigma}_t\big|\mathcal{G}_t\right].
    \end{equation*}
    In the following, we only prove the second inequality since the first one can be proved similarly. 

    For any $\sigma\in\mathcal{T}^{\mathbb{G}}_{t,T}$, Set $\hat{Y}=Y-y^{\tau^*_t, \sigma}$, $\hat{Z}=Z-z^{\tau^*_t, \sigma}$. It is easy to check that
    \begin{align*}
        \hat{Y}_s=\hat{Y}_{\tau^*_t\wedge \sigma}+\int_s^{\tau^*_t\wedge\sigma}(a_r \hat{Y}_r+b_r \hat{Z}_r)dr+K_{\tau^*_t\wedge \sigma}-K_s-\int_s^{\tau^*_t\wedge\sigma}\hat{Z}_rdB_r.
    \end{align*}
    Solving the above BSDE explicitly, we obtain that 
    \begin{equation}\label{e3.6}
        \Gamma_t \hat{Y}_t=\E\left[\Gamma_{\tau^*_t\wedge \sigma}\hat{Y}_{\tau^*_t\wedge \sigma}+\int_t^{\tau^*_t\wedge \sigma}\Gamma_sdK_s\Big|\mathcal{F}_t\right].
    \end{equation}
    Note that for any $s\in[t,\tau^*_t)$, we have $\E[Y_s-L_s|\mathcal{G}_s]>0$. It follows from the Skorokhod condition that $K^+_s=K^+_t$, for all $s\in[t,\tau^*_t)$. Therefore, we have 
    \begin{align*}
        \int_t^{\tau^*_t\wedge \sigma}\Gamma_sdK_s=-\int_t^{\tau^*_t\wedge \sigma}\Gamma_sdK^-_s\leq 0.
    \end{align*}
    Plugging the above inequality into \eqref{e3.6} yields that
    \begin{equation*}
        \Gamma_t \hat{Y}_t\leq \E\left[\Gamma_{\tau^*_t\wedge \sigma}\hat{Y}_{\tau^*_t\wedge \sigma}\big|\mathcal{F}_t\right].
    \end{equation*}
    Taking conditional expectations on both sides and noting that $\Gamma$ is $\mathbb{G}$-adapted, we have
    \begin{equation}\label{e3.7}
        \Gamma_t\E[\hat{Y}_t|\mathcal{G}_t]\leq \E\left[\Gamma_{\tau^*_t\wedge \sigma}\hat{Y}_{\tau^*_t\wedge \sigma}\big|\mathcal{G}_t\right]=\E\left[\E\left[\Gamma_{\tau^*_t\wedge \sigma}\hat{Y}_{\tau^*_t\wedge \sigma}\big|\mathcal{G}_{\tau^*_t\wedge \sigma}\right]\Big|\mathcal{G}_t\right]=\E\left[\Gamma_{\tau^*_t\wedge \sigma}\E\left[\hat{Y}_{\tau^*_t\wedge \sigma}\big|\mathcal{G}_{\tau^*_t\wedge \sigma}\right]\Big|\mathcal{G}_t\right].
    \end{equation}
    Recall that 
    \begin{align*}
        \hat{Y}_{\tau^*_t\wedge \sigma}=Y_{\tau^*_t\wedge \sigma}-\left(\xi I_{\{\sigma\wedge \tau^*_t=T\}}+L_{\tau^*_t} I_{\{\tau^*_t<T,\tau^*_t\leq \sigma\}}+U_\sigma I_{\{\sigma<\tau^*_t\}}\right).
    \end{align*}
    It is easy to check that 
    \begin{align*}
        \E\left[\hat{Y}_{\tau^*_t\wedge \sigma}\big|\mathcal{G}_{\tau^*_t\wedge \sigma}\right]\leq 0,
    \end{align*}
    which together with \eqref{e3.7} implies that
    \begin{align*}
        \Gamma_t\E[\hat{Y}_t|\mathcal{G}_t]\leq 0.
    \end{align*}
    This amounts to say that for any $t\in[0,T]$, we have
    \begin{align*}
        \E[Y_t|\mathcal{G}_t]\leq \E\left[y^{\tau^*_t\wedge\sigma}_t\big|\mathcal{G}_t\right],
    \end{align*}
    which is the desired result. The proof is complete.
\end{proof}

\begin{remark}
        Theorem \ref{thm3.2} extends Theorem 3.2 in \cite{HHL}  to the case of double reflections. An interesting question is that under which condition the solution $\Gamma$ to \eqref{e3.2} is $\mathbb{G}$-adapted. The readers may refer to Example 3.2 in \cite{HHL}.
\end{remark}

Since the constraints are not given pointwisely, we cannot expect to establish the pointwise comparison theorem for the general case (see Example 3.2 in \cite{HHLLW}). Motivated by Corollary 3.3 in \cite{HHL}, the objective is to compare the conditional expectation of the solution $Y$ with respect to the partial information. More precisely, with the help of Theorem \ref{thm3.2}, we may obtain the following comparison theorem for doubly conditional  reflected BSDEs.

\begin{corollary}\label{cor3.3}
    Suppose that $(\xi^i,L^i,U^i)$, $i=1,2$, satisfy Assumptions (H2)-(H3) and $\Gamma$ is $\mathbb{G}$-adpated, where $\Gamma$ is the solution to \eqref{e3.2}. Furthermore, we assume that $\E[\xi^1|\mathcal{G}_T]\geq \E[\xi^2|\mathcal{G}_T]$ and for any $t\in[0,T]$, $\E[\eta^1_t|\mathcal{G}_t]\geq \E[\eta^2_t|\mathcal{G}_t]$ for $\eta=c,L,U$.   Let $(Y^i,L^i,U^i)$, $i=1,2$, be the unique solution to the following doubly conditional reflected BSDE
    \begin{equation*}
\begin{cases}
Y^i_t=\xi^i+\int_t^T (a_sY^i_s+b_sZ^i_s+c^i_s)ds-\int_t^T Z^i_s dB_s+K^i_T-K^i_t, \\
\E[L^i_t|\mathcal{G}_t]\leq \E[Y^i_t|\mathcal{G}_t]\leq \E[U^i_t|\mathcal{G}_t], \\
K^i_t=K^{i,+}_t-K^{i,-}_t, \ K^{i,+},K^{i,-}\in \mathcal{A}^2_{\mathbb{G}},\\
\int_0^T \E[Y^i_t-L^i_t|\mathcal{G}_t]dK_t^{i,+}=\int_0^T \E[U^i_t-L^i_t|\mathcal{G}_t]dK^{i,-}_t=0.
\end{cases}
\end{equation*}
Then, for any $t\in[0,T]$, we have 
\begin{align*}
    \E[Y^1_t|\mathcal{G}_t]\geq \E[Y^2_t|\mathcal{G}_t].
\end{align*}
\end{corollary}

\begin{proof}
    For each fixed $t\in[0,T]$, for any given $\tau,\sigma\in\mathcal{T}^{\mathbb{G}}_{t,T}$, let $(y^{i,\tau,\sigma},z^{i,\tau,\sigma})$ be the unique solution to the following BSDE
\begin{align*}
    y^{i,\tau,\sigma}_s=&\xi^i I_{\{\sigma\wedge \tau=T\}}+L^i_\tau I_{\{\tau<T,\tau\leq \sigma\}}+U^i_\sigma I_{\{\sigma<\tau\}}-\int_s^{\tau\wedge\sigma}z^{^ii,\tau,\sigma}_rdB_r\\
    &+\int_s^{\tau\wedge\sigma}(a_r y^{i,\tau,\sigma}_r+b_r z^{i,\tau,\sigma}_r+c^i_r)dr.
\end{align*}
For simplicity, we omit the superscript $\tau,\sigma$ and set $\hat{y}_t=y^1_t-y^2_t$, $\hat{c}_t=c^1_t-c^2_t$. By a similar analysis as the proof of Theorem \ref{thm3.2}, we have
\begin{align}\label{e3.10}
    \Gamma_t \E[\hat{y}_t|\mathcal{G}_t]=\E\left[\Gamma_{\tau\wedge \sigma}\hat{y}_{\tau\wedge \sigma}+\int_t^{\tau\wedge \sigma}\Gamma_s\hat{c}_sds\Big|\mathcal{G}_t\right].
\end{align}
It is easy to check that 
\begin{equation}\label{e3.10'}
\begin{split}
    &\E[\Gamma_{\tau\wedge \sigma}\hat{y}_{\tau\wedge \sigma}|\mathcal{G}_t]=\E[\Gamma_{\tau\wedge \sigma}\E[\hat{y}_{\tau\wedge \sigma}|\mathcal{G}_{\tau\wedge\sigma}]|\mathcal{G}_t]\geq0,\\
    &\E\left[\int_t^{\tau\wedge \sigma}\Gamma_s\hat{c}_sds\Big|\mathcal{G}_t\right]=\E\left[\int_t^{\tau\wedge \sigma}\Gamma_s\E[\hat{c}_s|\mathcal{G}_s]ds\Big|\mathcal{G}_t\right]\geq 0.
\end{split}
\end{equation}
Combining Eqs. \eqref{e3.10} and \eqref{e3.10'} yields that $\E[y^1_t|\mathcal{G}_t]\geq \E[y^2_t|\mathcal{G}_t]$ for any $\tau,\sigma\in \mathcal{T}^{\mathbb{G}}_{t,T}$. By Theorem \ref{thm3.2}, for $i=1,2$, we have 
\begin{align*}
    \E[Y^i_t|\mathcal{G}_t]=\essinf_{\sigma\in\mathcal{T}^{\mathbb{G}}_{t,T}}\esssup_{\tau\in\mathcal{T}^{\mathbb{G}}_{t,T}}\E\left[y_t^{i,\tau,\sigma}|\mathcal{G}_t\right].
\end{align*}
The proof is complete.
\end{proof}

\section{Optimal switching under partial information}

In this section, we consider the starting and stopping problem in reversible investment under partial information $\mathbb{G}$, which satisfies Assumption \ref{assG}. Partial information usually occurs in a situation where only part of the overall information of the market can be accessed by the agent. For the reversible investment problem, roughly speaking, the objective is to find a sequence of $\mathbb{G}$-stopping times to maximize the overall profit, where any two successive stopping times represent the time to stop the production and the time to resume it, respectively. We first give the definition of such kind of stopping times. 

\begin{definition}
    A strategy is an increasing sequence of $\mathbb{G}$-stopping times $\delta=\{\tau_n\}_{n\geq 1}$. A strategy $\delta=\{\tau_n\}_{n\geq 1}$ is called admissible if $\lim_{n\rightarrow\infty}\tau_n=T$, $\P$-a.s. An admissible strategy is called finite if $\P(\omega:\tau_n(\omega)<T,\forall n\geq1)=0$. The collection of admissible (resp., finite admissible) strategies is denoted by $\mathfrak{D}_a$ (resp., $\mathfrak{D}$).
\end{definition}

Now, we describe the starting and stopping problems in more detail. There are two states of the system, for example, one represents  the ongoing productivity and the other is the frozen  productivity. The agent should choose when to switch between these two states. We assume that the agent has partial information $\mathbb{G}$, which means that her decision is a set of $\mathbb{G}$-stopping times. For any $n\geq 1$ and any given strategy $\delta=\{\tau_n\}_{n\geq 1}$, during the time period $[0,\tau_1]$ and $(\tau_{2n},\tau_{2n+1}]$, the production is open and the instantaneous profit is $\psi_1(t)$  while during the time period $(\tau_{2n+1},\tau_{2n+2}]$, the production is closed and the instantaneous profit is $\psi_2(t)$. For any $t\in[0,T]$, Let $D_t$ (resp., $a_t$) represents the sunk cost when the production is stopped (resp., starts). For any $t\in[0,T]$, set
\begin{displaymath}
    u_t=I_{[0,\tau_1]}(t)+\sum_{n\geq 1}I_{(\tau_{2n},\tau_{2n+1}]}(t),
\end{displaymath}
which amounts to say that $u_t=1$ if the production is open, and  $0$ otherwise. For any $t\in[0,T]$, let $\Phi(t,0)=\psi_2(t)$ and $\Phi(t,1)=\psi_1(t)$. Then, the value function is given by 
\begin{align*}
    J(\delta)=\E\left[\int_0^T\Phi(s,u_s)ds-\sum_{n\geq 1}\left(D_{\tau_{2n-1}} I_{\{\tau_{2n-1}<T\}}+a_{\tau_{2n}} I_{\{\tau_{2n}<T\}}\right)\right].
\end{align*}
We first propose the following assumptions for the instantaneous profit and the sunk cost throughout this section.

\begin{assumption}\label{assaD}
   \item[(1)] The processes $\psi_i$ belong to $\mathcal{H}^2$, $i=1,2$.

   \item[(2)] The processes $a,D\in\mathcal{S}^2$ are  nonnegative such that 
   \begin{align*}
        \inf_{(t,\omega)\in[0,T]\times\Omega}a_t(\omega)\wedge D_t(\omega)>0.
    \end{align*}
    Besides, $a$ and $-D$ satisfy the Mokobodski condition.
\end{assumption}

The objective is to calculate the maximal expected profit and to find the optimal strategies such that the maximal value is obtained. It is natural that the optimal strategies should be necessarily finite, otherwise the sunk cost would be infinite, which leads to the following result.

\begin{proposition}
    We have $\sup_{\delta\in\mathfrak{D}_a}J(\delta)=\sup_{\delta\in\mathfrak{D}}J(\delta)$.
\end{proposition}

The following result plays a key role in constructing the optimal strategies.

\begin{proposition}\label{y1y2}
    There exist two processes $Y^i\in \mathcal{S}^2$, $i=1,2$, such that for any $t\in[0,T]$,
    \begin{equation*}\begin{split}
        &\E[Y^1_t|\mathcal{G}_t]=\esssup_{\tau\in \mathcal{T}^{\mathbb{G}}_{t,T}}\E\left[\int_t^\tau \psi_1(r)dr+(Y^2_\tau-D_\tau)I_{\{\tau<T\}}\Big|\mathcal{G}_t\right],\\
        &\E[Y^2_t|\mathcal{G}_t]=\esssup_{\tau\in \mathcal{T}^{\mathbb{G}}_{t,T}}\E\left[\int_t^\tau \psi_2(r)dr+(Y^1_\tau-a_\tau)I_{\{\tau<T\}}\Big|\mathcal{G}_t\right].
    \end{split}\end{equation*}
\end{proposition}

\begin{proof}
    By Theorem \ref{thm2.7}, the following doubly conditional reflected BSDE admits a unique solution $(Y,Z,K)\in\mathcal{S}^2\times\mathcal{H}^2\times\mathcal{BV}_{\mathbb{G}}^2$
    \begin{equation}\label{MRBSDE}
\begin{cases}
Y_t=\int_t^T (\psi_1(s)-\psi_2(s))ds-\int_t^T Z_s dB_s+K_T-K_t,\ t\in[0,T], \\
-\E[D_t|\mathcal{G}_t]\leq \E[Y_t|\mathcal{G}_t]\leq \E[a_t|\mathcal{G}_t], \ t\in[0,T],\\
K_t=K^+_t-K^-_t,\ K^+,K^-\in \mathcal{A}^2_\mathbb{G}, \\
 \int_0^T \E[Y_t+D_t|\mathcal{G}_t]dK_t^+=\int_0^T \E[a_t-Y_t|\mathcal{G}_t]dK^-_t=0.
\end{cases}
\end{equation}

Now for any $t\in[0,T]$, we set
\begin{align*}
&Y^1_t=\E_t\left[\int_t^T\psi_1(s)ds+K^+_T-K^+_t\right],\\
&Y^2_t=\E_t\left[\int_t^T\psi_2(s)ds+K^-_T-K^-_t\right].
\end{align*}
It is easy to check that $Y_t=Y^1_t-Y^2_t$, $t\in[0,T]$ and 
\begin{align*}
    \int_0^T \psi_1(s)ds+K^+_T,\int_0^T \psi_2(s)ds+K^-_T\in L^2(\mathcal{F}_T).
\end{align*}
Then, applying the martingale representation theorem, there exist $Z^i\in\mathcal{H}^2$, $i=1,2$, such that 
\begin{align*}
    &\E_t\left[\int_0^T \psi_1(s)ds+K^+_T\right]=\E\left[\int_0^T \psi_1(s)ds+K^+_T\right]+\int_0^t Z^1_sdB_s,\\
    &\E_t\left[\int_0^T \psi_2(s)ds+K^-_T\right]=\E\left[\int_0^T \psi_2(s)ds+K^-_T\right]+\int_0^t Z^2_sdB_s.
\end{align*}
Therefore, we have
\begin{equation*}
\begin{cases}
Y^1_t=\int_t^T \psi_1(s)ds-\int_t^T Z^1_s dB_s+K^+_T-K^+_t,\ t\in[0,T], \\
\E[Y^1_t|\mathcal{G}_t]\geq \E[Y^2_t-D_t|\mathcal{G}_t], \ t\in[0,T],\\
 \int_0^T (\E[Y^1_t|\mathcal{G}_t]-\E[Y^2_t-D_t|\mathcal{G}_t])dK_t^+=0
\end{cases}
\end{equation*}
and 
\begin{equation*}
\begin{cases}
Y^2_t=\int_t^T \psi_2(s)ds-\int_t^T Z^2_s dB_s+K^-_T-K^-_t,\ t\in[0,T], \\
\E[Y^2_t|\mathcal{G}_t]\geq \E[Y^1_t-a_t|\mathcal{G}_t], \ t\in[0,T],\\
 \int_0^T (\E[Y^2_t|\mathcal{G}_t]-\E[Y^1_t-a_t|\mathcal{G}_t])dK_t^-=0.
\end{cases}
\end{equation*}
That is, $(Y^1,Z^1,K^+)$ (resp., $(Y^2,K^2,K^-)$) can be viewed as the solution to the conditional reflected BSDEs with terminal value $0$, generator $\psi_1$ (resp., $\psi_2$) and obstacle $Y^2-D$ (resp., $Y^1-a$). By Theorem 3.2 in \cite{HHL}, we obtain the desired result.  
\end{proof}

\begin{remark}\label{optimaltime}
    In fact, modifying the proof of Theorem 3.2 in \cite{HHL}, for any $\mathbb{G}$-stopping time $\sigma$, we have 
    \begin{equation}\label{optimal stopping}\begin{split}
        &\E[Y^1_\sigma|\mathcal{G}_\sigma]=\esssup_{\tau\in \mathcal{T}^{\mathbb{G}}_{\sigma,T}}\E\left[\int_\sigma^\tau \psi_1(r)dr+(Y^2_\tau-D_\tau)I_{\{\tau<T\}}\Big|\mathcal{G}_\sigma\right],\\
        &\E[Y^2_\sigma|\mathcal{G}_\sigma]=\esssup_{\tau\in \mathcal{T}^{\mathbb{G}}_{\sigma,T}}\E\left[\int_\sigma^\tau \psi_2(r)dr+(Y^1_\tau-a_\tau)I_{\{\tau<T\}}\Big|\mathcal{G}_\sigma\right],
    \end{split}\end{equation}
  where $\mathcal{T}_{\sigma,T}^{\mathbb{G}}$ is the collection of $\mathbb{G}$-stopping times $\tau$ with $\sigma\leq \tau\leq T$.   Moreover, the optimal stopping times of problems \eqref{optimal stopping} exist, which can be represented as follows respectively
    \begin{align*}
        &\tau^*_1=\inf\{s\geq \sigma: \E[Y^1_s|\mathcal{G}_s]=\E[Y^2_s-D_s|\mathcal{G}_s]\}\wedge T, \\
        &\tau^*_2=\inf\{s\geq \sigma: \E[Y^2_s|\mathcal{G}_s]=\E[Y^1_s-a_s|\mathcal{G}_s]\}\wedge T.
    \end{align*}
\end{remark}

Now, we can state the main result of this section.

\begin{proposition}\label{prop2.2}
    Let $Y^i$ be given as in Proposition \ref{y1y2}, we have $Y^1_0=\sup_{\delta\in\mathfrak{D}}J(\delta)$. Furthermore, the strategy $\delta^*=\{\tau_n^*\}_{n\geq 1}$ is optimal, where $\tau^*_0=0$ and for any $n\geq 1$
    \begin{align*}
        &\tau^*_{2n-1}=\inf\{s\geq \tau^*_{2n-2}: \E[Y^1_s|\mathcal{G}_s]=\E[Y^2_s-D_s|\mathcal{G}_s]\}\wedge T,\\
        &\tau^*_{2n}=\inf\{s\geq \tau^*_{2n-1}: \E[Y^2_s|\mathcal{G}_s]=\E[Y^1_s-a_s|\mathcal{G}_s]\}\wedge T.
    \end{align*}
\end{proposition}

\begin{proof}
  \textbf{Step 1.}  We first show that for any $\delta\in\mathfrak{D}$, we have $Y^1_0\geq J(\delta)$.  For any given a finite strategy $\delta=\{\tau_n\}_{n\geq 1}$, by Remark \ref{optimaltime}, we have
    \begin{align*}
        &Y^1_0=\E[Y^1_0]\geq \E\left[\int_0^{\tau_1}\psi_1(r)dr+(Y^2_{\tau_1}-D_{\tau_1})I_{\{\tau_1<T\}}\right]\\
        & \ \ \ \  =\E\left[\int_0^{\tau_1}\psi_1(r)dr+\E[Y^2_{\tau_1}|\mathcal{G}_{\tau_1}]I_{\{\tau_1<T\}}-D_{\tau_1}I_{\{\tau_1<T\}}\right],\\
        &\E[Y^2_{\tau_1}|\mathcal{G}_{\tau_1}]\geq \E\left[\int_{\tau_1}^{\tau_2}\psi_2(r)dr+(Y^1_{\tau_2}-a_{\tau_2})I_{\{\tau_2<T\}}\Big|\mathcal{G}_{\tau_1}\right].
    \end{align*}
    Noting the facts that $\{\tau_2<T\}\subset \{\tau_1<T\}$ and $\{\tau_1=T\}\subset\{\tau_2=T\}$,  we have
    \begin{align*}
        Y^1_0\geq &\E\left[\int_0^{\tau_1}\psi_1(r)dr+\int_{\tau_1}^{\tau_2}\psi_2(r)dr-D_{\tau_1}I_{\{\tau_1<T\}}-a_{\tau_2}I_{\{\tau_2<T\}}+Y^1_{\tau_2}I_{\{\tau_2<T\}}\right]\\
        =&\E\left[\int_{0}^{\tau_2}\Psi(r)dr-D_{\tau_1}I_{\{\tau_1<T\}}-a_{\tau_2}I_{\{\tau_2<T\}}+Y^1_{\tau_2}I_{\{\tau_2<T\}}\right].
    \end{align*}
    Now, following this procedure as many times as necessary yields that 
\begin{align*}
    Y^1_0\geq \E\left[\int_{0}^{\tau_{2n}}\Psi(r)dr-\sum_{k=1}^n \left(D_{\tau_{2k-1}}I_{\{\tau_{2k-1}<T\}}+a_{\tau_{2k}}I_{\{\tau_{2k}<T\}}\right)+Y^1_{\tau_{2n}}I_{\{\tau_{2n}<T\}}\right].
\end{align*}
Since the strategy is finite, letting $n$ go to infinity, we obtain that $Y^1_0\geq J(\delta)$.

    \textbf{Step 2.} Now, we show that $\delta^*$ is an optimal strategy and it is finite. By Remark \ref{optimaltime}, we have 
    \begin{align*}
        &Y^1_0=\E[Y^1_0]=\E\left[\int_0^{\tau^*_1}\psi_1(r)dr+(Y^2_{\tau^*_1}-D_{\tau^*_1})I_{\{\tau^*_1<T\}}\right],\\
        &\E[Y^2_{\tau^*_1}|\mathcal{G}_{\tau^*_1}]=\E\left[\int_{\tau^*_1}^{\tau^*_2}\psi_2(r)dr+(Y^1_{\tau^*_2}-a_{\tau^*_2})I_{\{\tau^*_2<T\}}\Big|\mathcal{G}_{\tau^*_1}\right].
    \end{align*}
    By a similar analysis as Step 1, we have 
    \begin{align*}
        Y^1_0=&\E\left[\int_0^{\tau^*_1}\psi_1(r)dr+\int_{\tau^*_1}^{\tau^*_2}\psi_2(r)dr-D_{\tau^*_1}I_{\{\tau^*_1<T\}}-a_{\tau^*_2}I_{\{\tau^*_2<T\}}+Y^1_{\tau^*_2}I_{\{\tau^*_2<T\}}\right]\\
        =&\E\left[\int_{0}^{\tau^*_2}\Psi(r)dr-D_{\tau^*_1}I_{\{\tau^*_1<T\}}-a_{\tau^*_2}I_{\{\tau^*_2<T\}}+Y^1_{\tau^*_2}I_{\{\tau^*_2<T\}}\right].
    \end{align*}
    Repeating this procedure, for any $n\geq 1$, we obtain that 
    \begin{equation}\label{y10}
        Y^1_0=\E\left[\int_{0}^{\tau^*_{2n}}\Psi(r)dr-\sum_{k=1}^n \left(D_{\tau^*_{2k-1}}I_{\{\tau^*_{2k-1}<T\}}+a_{\tau^*_{2k}}I_{\{\tau^*_{2k}<T\}}\right)+Y^1_{\tau^*_{2n}}I_{\{\tau^*_{2n}<T\}}\right].
    \end{equation}
    We claim that the strategy is finite. Otherwise, suppose that $\P(A)>0$, where $A=\{\omega:\tau^*_n(\omega)<T,\forall n\geq1\}$. For any $n\geq 1$, simple calculation yields that 
    \begin{align*}
        Y^1_0&\leq \E\Bigg[\int_{0}^T(|\psi_1(r)|\vee|\psi_2(r)|)dr-\sum_{k=1}^n \left(D_{\tau^*_{2k-1}}I_{\{\tau^*_{2k-1}<T\}}+a_{\tau^*_{2k}}I_{\{\tau^*_{2k}<T\}}\right)I_A\\
        &\ \ \ \ \ \ \ \ +\sup_{s\in[0,T]}|Y^1_s|-\sum_{k=1}^n \left(D_{\tau^*_{2k-1}}I_{\{\tau^*_{2k-1}<T\}}+a_{\tau^*_{2k}}I_{\{\tau^*_{2k}<T\}}\right)I_{A^c}\Bigg]\\
        &\leq C-\E\left[\sum_{k=1}^n \left(D_{\tau^*_{2k-1}}I_{\{\tau^*_{2k-1}<T\}}+a_{\tau^*_{2k}}I_{\{\tau^*_{2k}<T\}}\right)I_A\right]\rightarrow-\infty, \textrm{ as } n\rightarrow\infty,
    \end{align*}
    which is a contradiction. Now, letting $n$ approach infinity in \eqref{y10}, we obtain that $Y^1_0=J(\delta^*)$. The proof is complete.
\end{proof}

\begin{remark}
    When $\mathbb{G}=\mathbb{F}$, Proposition \ref{y1y2} and Proposition \ref{prop2.2} degenerates into Theorem 3.2 and Proposition 2.2 in \cite{HJ}, respectively.
\end{remark}

\section*{Acknowledgments}
	This work was supported  by the National Natural Science Foundation of China (No. 12301178), the Natural Science Foundation of Shandong Province for Excellent Young Scientists Fund Program (Overseas) (No. 2023HWYQ-049), the Natural Science Foundation of Shandong Province (No. ZR2023ZD35) and  the Qilu Young Scholars Program of Shandong University. 

\appendix
\renewcommand\thesection{Appendix}
\section{ }
\renewcommand\thesection{A}
\setcounter{equation}{0}
\renewcommand{\theequation}{A\arabic{equation}}

In this section, we construct the solution to the doubly conditional reflected BSDEs \eqref{nonlinearyz} via the penalization method. More precisely,  consider the following family of conditional expectation BSDEs parameterized by $n$
\begin{equation}\begin{split}\label{panelization}
Y_t^n=&\xi+\int_t^T f(s,Y_s^n,Z_s^n)ds
+\int_t^T n(\E[Y_s^n-L_s|\mathcal{G}_s])^-ds\\
&-\int_t^T n(\E[Y_s^n-U_s|\mathcal{G}_s])^+ds-\int_t^T Z_s^n dB_s.
\end{split}\end{equation}
Throughout this section, $\mathbb{G}=\{\mathcal{G}_t\}_{t\in[0,T]}$ is a subfiltration of $\mathbb{F}$ without any continuity assumption as in Assumption \ref{assG}. Suppose that $f$ satisfies (H1), $\xi\in L^2(\mathcal{F}_T)$ and $L,U\in \mathcal{H}^2$ with $L\leq U$. By Theorem 2.7 in in \cite{Li}, for each $n\in\mathbb{N}$, there exists a unique pair of solution $(Y^n,Z^n)\in \mathcal{S}^2\times \mathcal{H}^2$ to the above equation. Set $K^n=K^{n,+}-K^{n,-}$, where 
$$K^{n,+}_t=\int_0^t n(\E[Y_s^n-L_s|\mathcal{G}_s])^- ds, \ K^{n,-}_t=\int_0^t n(\E[Y_s^n-U_s|\mathcal{G}_s])^- ds.$$  
We claim that $(Y^n,Z^n,K^n)$ converges to $(Y,Z,K)$ and the triple of  limiting processes is indeed the solution to the doubly conditional reflected BSDE \eqref{nonlinearyz}. It should be pointed out that, different from the penalized conditional expectation BSDEs for single reflected case studied in \cite{Li}, $K^n$ is not monotone in the present framework. This leads to the difficulty to  derive the uniform $\mathcal{S}^2$-estimates for $\{Y^n\}$ and $\{K^n\}$, respectively. To this end, we first show that $\{Y^n\}$ is uniformly bounded under appropriate ``norm".  In the sequel, $C$ will always represent a constant independent of $n$, which may vary from line to line. 

\begin{lemma}\label{estimateYnZn}
Given $\xi\in L^2(\mathcal{F}_T)$ and $L,U\in\mathcal{H}^2$, suppose that $f$ satisfies (H1). We assume that there exists an It\^{o} process $I$ with representation
\begin{align*}
    I_t=I_0+\int_0^t b_s ds+\int_0^t \sigma_sdB_s,
\end{align*}
where $b,\sigma\in\mathcal{H}^2$, such that for any $t\in[0,T]$, $\E[L_t|\mathcal{G}_t]\leq \E[I_t|\mathcal{G}_t]\leq \E[U_t|\mathcal{G}_t]$. Then, we have
\begin{equation}\label{estimate1}
\sup_{t\in[0,T]}\E\left[|Y_t^n|^2|\right]+\E\left[\int_0^T |Z_t^n|^2dt\right]\leq C\E\left[|\xi|^2+\sup_{t\in[0,T]}|I_t|^2+\int_0^T \left(|f(s,0,0)|^2+|b_s|^2+|\sigma_s|^2\right)ds\right].
\end{equation}
Moreover, suppose that 
there exists a constant $p>2$, such that $\xi\in L^p(\mathcal{F}_T)$ and $b,\sigma,f(\cdot,0,0)\in \mathcal{H}^p$, i.e.,
    \begin{align*}
        &\E[|\xi|^p]<\infty, \ \E\left[\left(\int_0^T |f(s,0,0)|^2ds\right)^{\frac{p}{2}}\right]<\infty, \\
        &\E\left[\left(\int_0^T |b_s|^2ds\right)^{\frac{p}{2}}\right]<\infty, \ \E\left[\left(\int_0^T |\sigma_s|^2ds\right)^{\frac{p}{2}}\right]<\infty.
    \end{align*}
Then, we have 
\begin{equation}\label{estimate2}
    \E\left[\sup_{t\in[0,T]}\left(\E\left[|Y^n_t|^2|\mathcal{G}_t\right]\right)^{\frac{p}{2}}\right]\leq C\E\left[|\xi|^p+\sup_{t\in[0,T]}|I_t|^p+\left(\int_0^T \left(|f(s,0,0)|^2+|b_s|^2+|\sigma_s|^2\right)ds\right)^{\frac{p}{2}}\right].
\end{equation}
\end{lemma}

\begin{proof}
Set $\bar{Y}^n_t=Y^n_t-I_t$, $\bar{Z}^n_t=Z^n_t-\sigma_t$, $\bar{U}_t=U_t-I_t$, $\bar{L}_t=L_t-I_t$ and $\bar{f}^n_t=f(t,Y^n_t,Z^n_t)+b_t$. For any fixed positive constant $\beta$, applying It\^{o}'s formula to $e^{\beta t}|\bar{Y}^n_t|^2$,  we have
\begin{equation}\label{eq1.61}\begin{split}
&e^{\beta t}|\bar{Y}^n_t|^2+\int_t^T \beta e^{\beta s}|\bar{Y}^n_s|^2 ds+\int_t^T e^{\beta s}|\bar{Z}_s^n|^2 ds\\
=&e^{\beta T}|\xi-I_T|^2+\int_t^T 2e^{\beta s}\bar{Y}^n_s\bar{f}^n_sds-\int_t^T 2e^{\beta s}\bar{Y}_s^n \bar{Z}_s^ndB_s\\
&+\int_t^T 2ne^{\beta s}\bar{Y}_s^n(\E[\bar{Y}_s^n-\bar{L}_s|\mathcal{G}_s])^-ds-\int_t^T 2ne^{\beta s}\bar{Y}_s^n(\E[\bar{Y}_s^n-\bar{U}_s|\mathcal{G}_s])^+ds.
\end{split}\end{equation}
It is easy to check that 
\begin{equation}\label{eq1.62}\begin{split}
2\bar{Y}^n_s\bar{f}^n_s
\leq |f(s,0,0)|^2+\frac{1}{2}|\bar{Z}^n_s|^2+|b_s|+|\sigma_s|^2+|I_s|^2+(2+2\lambda+4\lambda^2)|\bar{Y}^n_s|^2.
\end{split}\end{equation}
Noting that $\E[\bar{L}_s|\mathcal{G}_s]\leq 0$, simple calculation yields that
\begin{equation}\label{eq1.62'}
    \begin{split}
        &n\E\left[\int_t^T e^{\beta s}\bar{Y}_s^n(\E[\bar{Y}_s^n-\bar{L}_s|\mathcal{G}_s])^-ds\Big|\mathcal{G}_t\right]\\
        =&n\int_t^T\E\left[e^{\beta s}\bar{Y}_s^n(\E[\bar{Y}_s^n-\bar{L}_s|\mathcal{G}_s])^-|\mathcal{G}_t\right]ds\\
        =&n\int_t^T\E\left[e^{\beta s}\E[\bar{Y}_s^n|\mathcal{G}_s](\E[\bar{Y}_s^n-\bar{L}_s|\mathcal{G}_s])^-|\mathcal{G}_t\right]ds\leq 0.
    \end{split}
\end{equation}
Similarly, we have 
\begin{equation}\label{eq1.62''''}
    n\E\left[\int_t^T e^{\beta s}\bar{Y}_s^n(\E[\bar{Y}_s^n-\bar{U}_s|\mathcal{G}_s])^+ds\Big|\mathcal{G}_t\right]\geq 0.
\end{equation}
Set $\beta=2+2\lambda+4\lambda^2$. Plugging Eq. \eqref{eq1.62} into Eq. \eqref{eq1.61}, taking conditional expectations w.r.t. $\mathcal{G}_t$ on both sides and noting \eqref{eq1.62'}-\eqref{eq1.62''''}, we have
\begin{equation*}
\E\left[|\bar{Y}^n_t|^2+\int_t^T |\bar{Z}_s^n|^2 ds\Big|\mathcal{G}_t\right]\leq C\E\left[|\xi|^2+\sup_{s\in[0,T]}|I_s|^2+\int_0^T \left(|f(s,0,0)|^2+|\sigma_s|^2+|b_s|^2\right)ds\Big|\mathcal{G}_t\right].
\end{equation*}
Recalling the definition of $\bar{Y}^n$ and $\bar{Z}^n$, it follows that 
\begin{equation}\label{ynzn}
\E\left[|{Y}^n_t|^2+\int_t^T |{Z}_s^n|^2 ds\Big|\mathcal{G}_t\right]\leq C\E\left[|\xi|^2+\sup_{s\in[0,T]}|I_s|^2+\int_0^T \left(|f(s,0,0)|^2+|\sigma_s|^2+|b_s|^2\right)ds\Big|\mathcal{G}_t\right].
\end{equation}
Clearly,  estimate \eqref{estimate1} holds. By Eq. \eqref{ynzn} and Doob's maximal inequality, we obtain estimate \eqref{estimate2}.
\end{proof}

\begin{remark}
    Suppose that $\mathbb{G}=\mathbb{F}$, that is, consider the doubly reflected BSDE and its penalized BSDEs. It is easy to check that  
    \begin{align*}
        n\int_t^T e^{\beta s}\bar{Y}_s^n(\bar{Y}_s^n-\bar{L}_s)^-ds\leq 0, \ n\int_t^T e^{\beta s}\bar{Y}_s^n(\bar{Y}_s^n-\bar{U}_s)^-ds\geq 0. 
    \end{align*}
    Therefore, choosing $\beta=2+2\lambda+4\lambda^2$, Eq. \eqref{eq1.61} directly implies that 
    \begin{align*}
        \sup_{t\in[0,T]}|\bar{Y}^n_t|^2\leq C\left\{|\xi|^2+\sup_{s\in[0,T]}|I_s|^2+\int_0^T \left(|f(s,0,0)|^2+|\sigma_s|^2+|b_s|^2\right)ds\right\}.
    \end{align*}
    Then, we may obtain the uniform $\mathcal{S}^2$-estimate for $\{Y^n\}$. However, the main difficulty for the case of conditional reflection is that only Eqs. \eqref{eq1.62'}-\eqref{eq1.62''''} hold, which cannot be applied directly to establish the uniform $\mathcal{S}^2$-estimate for $\{Y^n\}$.
\end{remark}


In the following, we show that the running supremum (resp., infimum) of the negative (resp., positive) part of the conditional expectation $\E[Y^n_t-L_t|\mathcal{G}_t]$ (resp., $\E[Y^n_t-U_t|\mathcal{G}_t]$) converges to $0$ under the norm $\|\cdot\|_{\mathcal{S}^2}$. Moreover, the explicit convergence rate is established under some additional  assumptions for the driver $f$ and the obstacles $L,U$ (similar assumptions can be found in \cite{EKPPQ} for the doubly reflected BSDEs). This result plays a key role both in  establishing the uniform $\mathcal{S}^2$-estimates for $\{Y^n\}$, $\{K^{n,+}\}$ and $\{K^{n,-}\}$ (see Corollary \ref{cor1} below) and in proving the convergence result for $\{Y^n\}$, $\{Z^n\}$ and $\{K^n\}$ (see Theorem \ref{main2} below).

\begin{itemize}
    \item[(H1')] $f$ satisfies either of the following conditions: 
    \begin{itemize}
    \item for any $(\omega,t,z)\in\Omega\times[0,T]\times\mathbb{R}^d$, there exists a constant $M>0$ such that $|f(\omega,t,0,z)|\leq M$; 
    \item $f$ is independent of $z$ and 
    \begin{align*}
        \E\left[\sup_{t\in[0,T]}|f(t,0)|^2\right]<\infty.
    \end{align*}
    \end{itemize}
    \item[(H2')] $L,S$ are It\^{o} processes with representation
    \begin{align*}
        &L_t=L_0+\int_0^t b^L_s ds+\int_0^t \sigma^L_s dB_s,\\
        &U_t=U_0+\int_0^t b^U_s ds+\int_0^t \sigma^U_s dB_s,
    \end{align*}
    where $b^L,b^U\in\mathcal{S}^2$ and $\sigma^L,\sigma^U\in\mathcal{H}^2$. Moreover, for any $t\in[0,T]$, we have $\E[L_t|\mathcal{G}_t]\leq \E[U_t|\mathcal{G}_t]$. 
\end{itemize}

\begin{lemma}\label{estimateYn-S}
Let all assumptions in Lemma \ref{estimateYnZn} hold. Under (H1'), (H2') and (H3), there exists a constant $C$ independent of $n$, such that
\begin{align*}
\E\left[\sup_{t\in[0,T]}|(\E[Y_t^n-L_t|\mathcal{G}_t])^-|^2 \right]\leq \frac{C}{n^2}, \ \E\left[\sup_{t\in[0,T]}|(\E[Y_t^n-U_t|\mathcal{G}_t])^+|^2 \right]\leq \frac{C}{n^2}.
\end{align*}
\end{lemma}

\begin{proof}
  We only prove the first inequality since the second one can be proved similarly.  Without loss of generality, suppose that $f$ is independent of $z$ and 
    \begin{align*}
        \E\left[\sup_{t\in[0,T]}|f(t,0)|^2\right]<\infty.
    \end{align*}
    Set $\widetilde{Y}_t^n=Y^n_t-L_t$.
   Applying It\^{o}'s formula to $e^{-nt}\widetilde{Y}^n_t$, 
we obtain that 
\begin{equation}\label{eq1.63'}\begin{split}
\widetilde{Y}^n_t=&(\xi-L_T)e^{n(t-T)}+\int_t^T n e^{n(t-s)}\left[\widetilde{Y}^n_s+\left(\E\left[\widetilde{Y}^n_s\big|\mathcal{G}_s\right]\right)^-\right]ds\\
&+\int_t^T e^{n(t-s)}(f(s,Y^n_s)+b^L_s)ds-\int_t^T\widetilde{Z}^n_sdB_s\\
\geq &(\xi-L_T)e^{n(t-T)}+\int_t^T n e^{n(t-s)}\left[\widetilde{Y}^n_s+\left(\E\left[\widetilde{Y}^n_s\big|\mathcal{G}_s\right]\right)^-\right]ds\\
&-\int_t^T e^{n(t-s)}(|b^L_s|+|f(s,0)|+\lambda|Y^n_s|)ds-\int_t^T\widetilde{Z}^n_sdB_s,
\end{split}\end{equation} 
where $\widetilde{Z}^n_t=Z^n_t-\sigma^L_t$. It is easy to check that  $\E[\xi-L_T|\mathcal{G}_t]=\E[\E[\xi-L_T|\mathcal{G}_T]|\mathcal{G}_t]\geq 0$ and  
\begin{align*}
 &\E\left[\int_t^T n e^{n(t-s)}\left[\widetilde{Y}^n_s+\left(\E\left[\widetilde{Y}^n_s\big|\mathcal{G}_s\right]\right)^-\right]ds\Big|\mathcal{G}_t\right]\\   
 =&\int_t^T n e^{n(t-s)}\E\left[\widetilde{Y}^n_s+\left(\E\left[\widetilde{Y}^n_s\big|\mathcal{G}_s\right]\right)^-\Big|\mathcal{G}_t\right]ds\\
 =&\int_t^T n e^{n(t-s)}\E\left[\E\left[\widetilde{Y}^n_s\big|\mathcal{G}_s\right]+\left(\E\left[\widetilde{Y}^n_s\big|\mathcal{G}_s\right]\right)^-\Big|\mathcal{G}_t\right]ds\geq 0.
\end{align*} 
Taking conditional expectation w.r.t. $\mathcal{G}_t$ on both sides of  \eqref{eq1.63'} yields  that 
\begin{equation*}\begin{split}
\E\left[\widetilde{Y}^n_t\big|\mathcal{G}_t\right]
\geq 
-\E\left[\int_t^T e^{n(t-s)}(|b^L_s|+|f(s,0)|+\lambda|Y^n_s|)ds\Big|\mathcal{G}_t\right],
\end{split}\end{equation*} 
Consequently, we have
\begin{align*}
    \left(\E\left[\widetilde{Y}^n_t\big|\mathcal{G}_t\right]\right)^-
&\leq \E\left[\int_t^T e^{n(t-s)}(|b^L_s|+|f(s,0)|+\lambda|Y^n_s|)ds\Big|\mathcal{G}_t\right]\\
&= \E\left[\int_t^T e^{n(t-s)}\left(|b^L_s|+|f(s,0)|+\lambda\E\left[|Y^n_s||\mathcal{G}_s\right]\right)ds\Big|\mathcal{G}_t\right]\\
&\leq \frac{1}{n}\E\left[\sup_{s\in[0,T]}(|b^L_s|+|f(s,0)|+\lambda\E[|Y^n_s||\mathcal{G}_s])\Big|\mathcal{G}_t\right].
\end{align*}
Applying Eq. \eqref{estimate2} and Doob's maximal inequality, we obtain the desired result.
\end{proof}

\begin{corollary}\label{cor1}
    Under the same assumption as in Lemma \ref{estimateYn-S}, there exists a constant $C$ independent of $n$, such that
    \begin{align*}
      \E\left[|K^{n,+}_T|^2\right]\leq C, \ \E\left[|K^{n,-}_T|^2\right]\leq C \textrm{ and }  \E\left[\sup_{t\in[0,T]}|Y^n_t|^2\right]\leq C.
    \end{align*}
\end{corollary}

\begin{proof}
   The first two estimates are the direct consequence of Lemma \ref{estimateYn-S}. 
    By Eq. \eqref{panelization}, the H\"{o}lder inequality and the BDG inequality, we have 
    \begin{align*}
        \E\left[\sup_{t\in[0,T]}|Y^n_t|^2\right]\leq &C\Bigg\{\E\left[|\xi|^2+|K^{n,+}_T|^2+|K^{n,-}_T|^2\right]+\sup_{t\in[0,T]}\E[|Y^n_t|^2]\\
        &+\E\left[\int_0^T|Z^n_t|^2dt+\int_0^T|f(t,0,0)|^2dt\right]\Bigg\}.
    \end{align*}
    By Eq. \eqref{estimate1} and the uniform estimate for $K^{n,+}$, $K^{n,-}$, we obtain the last estimate.
\end{proof}


Now, we state the main result in this section.
\begin{theorem}\label{main2}
Under the same assumption as in Lemma \ref{estimateYn-S}, the BSDE with conditional reflection \eqref{nonlinearyz} has a unique solution $(Y,Z,K)$. Furthermore, $(Y,Z,K)$ is the limit of $(Y^n,Z^n,K^n)$.
\end{theorem}

   

\begin{proof}
   We only prove that the limit of $(Y^n,Z^n,K^n)$ is the solution to the BSDE with conditional reflection \eqref{nonlinearyz}. For simplicity, we write $Y^{m,n}_t=Y^m_t-Y^n_t$, $Z^{m,n}_t=Z^m_t-Z^n_t$, $K^{m,n}_t=K^m_t-K^n_t$, $K^{m,n,+}_t=K^{m,+}_t-K^{n,+}_t$, $K^{m,n,-}_t=K^{m,-}_t-K^{n,-}_t$ and $f^{m,n}_t=f(t,Y^m_t,Z^m_t)-f(t,Y^n_t,Z^n_t)$. 
    
  \textbf{Step 1.} We first show that  
  \begin{equation}\label{eq1.65}
    \lim_{m,n\rightarrow \infty} \sup_{t\in[0,T]}\E\left[|Y^{m,n}_t|^2\right]=0, \ \lim_{m,n\rightarrow\infty}\E\left[\int_0^T |Z^{m,n}_s|^2 ds\right]=0.
\end{equation}
 To this end,   applying It\^{o}'s formula to $e^{\beta t}|Y^{m,n}_t|^2$, we have
    \begin{equation}\label{eq1.64}\begin{split}
        &e^{\beta t}|{Y}^{m,n}_t|^2+\int_t^T \beta e^{\beta s}|{Y}^{m,n}_s|^2 ds+\int_t^T e^{\beta s}|Z_s^{m,n}|^2 ds\\
=&\int_t^T 2e^{\beta s}{Y}^{m,n}_sf^{m,n}_sds-\int_t^T 2e^{\beta s}{Y}_s^{m,n} Z_s^{m,n}dB_s\\
&+\int_t^T 2e^{\beta s}{Y}_s^{m,n}dK^{m,n,+}_s-\int_t^T 2e^{\beta s}{Y}_s^{m,n}dK^{m,n,-}_s\\
\leq &2(\lambda+\lambda^2)\int_t^T e^{\beta s}|Y^{m,n}_s|^2ds-\int_t^T 2e^{\beta s}{Y}_s^{m,n} Z_s^{m,n}dB_s \\
&+\frac{1}{2}\int_t^T e^{\beta s}|Z^{m,n}|^2 ds+\int_t^T 2e^{\beta s}{Y}_s^{m,n}dK^{m,n,+}_s-\int_t^T 2e^{\beta s}{Y}_s^{m,n}dK^{m,n,-}_s.
    \end{split}\end{equation}
Simple calculation yields that for any $u\in[0,t]$, 
\begin{equation*}\begin{split}
    &\E\left[\int_t^T e^{\beta s}{Y}_s^{m,n}dK^{m,n,+}_s\Big|\mathcal{G}_u\right]\\
    =&\int_t^T e^{\beta s}\E\left[\left(\widetilde{Y}^m_s-\widetilde{Y}^n_s\right)\left(m\left(\E\left[\widetilde{Y}^m_s\big|\mathcal{G}_s\right]\right)^--n\left(\E\left[\widetilde{Y}^n_s\big|\mathcal{G}_s\right]\right)^-\right)\Big|\mathcal{G}_u\right]ds\\
    =&\int_t^T e^{\beta s}\E\left[\left(\E\left[\widetilde{Y}^m_s\big|\mathcal{G}_s\right]-\E\left[\widetilde{Y}^n_s\big|\mathcal{G}_s\right]\right)\left(m\left(\E\left[\widetilde{Y}^m_s\big|\mathcal{G}_s\right]\right)^--n\left(\E\left[\widetilde{Y}^n_s\big|\mathcal{G}_s\right]\right)^-\right)\Big|\mathcal{G}_u\right]ds\\
    \leq &(m+n)\E\left[\int_t^T e^{\beta s}\left(\E\left[\widetilde{Y}^m_s\big|\mathcal{G}_s\right]\right)^-\left(\E\left[\widetilde{Y}^n_s\big|\mathcal{G}_s\right]\right)^- \Big|\mathcal{G}_u\right]\\
    \leq &e^{\beta T}\E\left[\sup_{s\in[0,T]} \left(\E\left[\widetilde{Y}^n_s\big|\mathcal{G}_s\right]\right)^-K^{m,+}_T+\sup_{s\in[0,T]} \left(\E\left[\widetilde{Y}^m_s\big|\mathcal{G}_s\right]\right)^-K^{n,+}_T\Big|\mathcal{G}_u\right],
    \end{split}
\end{equation*}
where $\widetilde{Y}_t^i=Y^i_t-L_t$ for $i=m,n$. Similarly, we have
\begin{equation*}
    \begin{split}
        &-\E\left[\int_t^T e^{\beta s}{Y}_s^{m,n}dK^{m,n,-}_s\Big|\mathcal{G}_u\right]\\
        \leq &e^{\beta T}\E\left[\sup_{s\in[0,T]} \left(\E\left[\overline{Y}^n_s\big|\mathcal{G}_s\right]\right)^+K^{m,-}_T+\sup_{s\in[0,T]} \left(\E\left[\overline{Y}^m_s\big|\mathcal{G}_s\right]\right)^+K^{n,-}_T\Big|\mathcal{G}_u\right],
    \end{split}
\end{equation*}
where $\overline{Y}^i_t=Y^i_t-U_t$ for $i=m,n$. 
In the following of the proof, we choose $\beta=2(\lambda+\lambda^2)$.  Taking expectations on both sides of \eqref{eq1.64}, all the above analysis indicates that for any $t\in[0,T]$, 
\begin{align*}
    &\E\left[|Y^{m,n}_t|^2\right]+\E\left[\int_t^T |Z^{m,n}_s|^2 ds\right]\\
    \leq &C\E\left[\sup_{s\in[0,T]} \left(\E\left[\widetilde{Y}^n_s\big|\mathcal{G}_s\right]\right)^-K^{m,+}_T+\sup_{s\in[0,T]} \left(\E\left[\widetilde{Y}^m_s\big|\mathcal{G}_s\right]\right)^-K^{n,+}_T\right]\\
    &+C\E\left[\sup_{s\in[0,T]} \left(\E\left[\overline{Y}^n_s\big|\mathcal{G}_s\right]\right)^+K^{m,-}_T+\sup_{s\in[0,T]} \left(\E\left[\overline{Y}^m_s\big|\mathcal{G}_s\right]\right)^+K^{n,-}_T\right].
\end{align*}
Then, \eqref{eq1.65} follows from the H\"{o}lder inequality, Lemma \ref{estimateYn-S} and Corollary \ref{cor1}.

\textbf{Step 2.} We show that 
\begin{align*}
        \lim_{m,n\rightarrow \infty}\E\left[\sup_{t\in[0,T]}|Y^{m,n}_t|^2\right]=0, \ \lim_{m,n\rightarrow \infty}\E\left[\sup_{t\in[0,T]}|K^{m,n}_t|^2\right]=0.
    \end{align*} 
Recalling \eqref{eq1.64} and $\beta=2(\lambda+\lambda^2)$, it is easy to check that 
\begin{equation}\label{supymn}\begin{split}
    \sup_{t\in[0,T]}|Y^{m,n}_t|^2\leq &C\Bigg(\sup_{t\in[0,T]}\Big|\int_t^T {Y}_s^{m,n} Z_s^{m,n}dB_s\Big|+\int_0^T|Y^{m,n}_s|d(K^{m,+}_s+K^{n,+}_s+K^{m,-}_s+K^{n,-}_s)\Bigg)\\
    \leq &C\Bigg(\int_0^T |{Y}_s^{m,n}|ds \left[\sup_{s\in[0,T]}\left(m\left(\E\left[\widetilde{Y}^m_s\big|\mathcal{G}_s\right]\right)^-+n\left(\E\left[\widetilde{Y}^n_s\big|\mathcal{G}_s\right]\right)^-\right)\right]\\
    &+\int_0^T |{Y}_s^{m,n}|ds \left[\sup_{s\in[0,T]}\left(m\left(\E\left[\overline{Y}^m_s\big|\mathcal{G}_s\right]\right)^-+n\left(\E\left[\overline{Y}^n_s\big|\mathcal{G}_s\right]\right)^-\right)\right]\\
    &\ \ \ \ \ \ +\sup_{t\in[0,T]}\Big|\int_t^T {Y}_s^{m,n} Z_s^{m,n}dB_s\Big|\Bigg).
\end{split}\end{equation}
It follows from Lemma \ref{estimateYn-S} and the H\"{o}lder inequality that, for $i=m,n$,
\begin{equation}\label{ymn}\begin{split}
&\E\left[\sup_{s\in[0,T]}i\left(\E\left[\widetilde{Y}^i_s\big|\mathcal{G}_s\right]\right)^-\int_0^T |{Y}_s^{m,n}|ds \right]\leq  C\left(\E\left[\int_0^T |{Y}_s^{m,n}|^2ds\right]\right)^{1/2},\\
&\E\left[\sup_{s\in[0,T]}i\left(\E\left[\overline{Y}^i_s\big|\mathcal{G}_s\right]\right)^+\int_0^T |{Y}_s^{m,n}|ds \right]\leq  C\left(\E\left[\int_0^T |{Y}_s^{m,n}|^2ds\right]\right)^{1/2}.
\end{split}\end{equation}
Applying the BDG inequality, we have 
\begin{equation}\begin{split}\label{eq1.65'}
    \E\left[\sup_{t\in[0,T]}\Big|\int_t^T {Y}_s^{m,n} Z_s^{m,n}dB_s\Big|\right]\leq& C\E\left[\left(\int_0^T |Y^{m,n}_s Z^{m,n}_s|^2ds\right)^{1/2}\right]\\
    \leq &C\E\left[\sup_{t\in[0,T]}|Y^{m,n}_t|\left(\int_0^T | Z^{m,n}_s|^2ds\right)^{1/2}\right]\\
    \leq& C\left(\E\left[\sup_{t\in[0,T]}|Y^{m,n}_t|^2\right]\right)^{1/2}\left(\E\left[\int_0^T | Z^{m,n}_s|^2ds\right]\right)^{1/2}\\
    \leq &\varepsilon \E\left[\sup_{t\in[0,T]}|Y^{m,n}_t|^2\right]+C\E\left[\int_0^T | Z^{m,n}_s|^2ds\right].
    \end{split}
\end{equation}
Choosing $\varepsilon<1$, combining Eqs. \eqref{supymn}-\eqref{eq1.65'} indicates that 
\begin{align*}
    \E\left[\sup_{t\in[0,T]}|Y^{m,n}_t|^2\right]\leq C\left(\left(\E\left[\int_0^T |{Y}_s^{m,n}|^2ds\right]\right)^{1/2}+\E\left[\int_0^T | Z^{m,n}_s|^2ds\right]\right).
\end{align*}
It follows from \eqref{eq1.65} that $\{Y^n\}_{n\in\mathbb{N}}$ is a Cauchy sequenc in $\mathcal{S}^2$. 
Finally, note that 
\begin{align*}
    K^{m,n}_t=Y^{m,n}_0-Y^{m,n}_t+\int_0^t Z^{m,n}_sdB_s-\int_0^t f^{m,n}_sds.
\end{align*}
Simple calculation yields that 
\begin{align*}
    \lim_{m,n\rightarrow \infty}\E\left[\sup_{t\in[0,T]}|K^m_t-K^n_t|^2\right]=0.
\end{align*}

\textbf{Step 3.} Let $(Y,Z,K)$ be a triple of processes, such that 
 \begin{align*}
        &\lim_{n\rightarrow \infty}\E\left[\sup_{t\in[0,T]}|Y_t-Y^n_t|^2\right]=0,\\ &\lim_{n\rightarrow \infty}\E\left[\int_0^T|Z_t-Z^n_t|^2dt\right]=0, \\ &\lim_{n\rightarrow \infty}\E\left[\sup_{t\in[0,T]}|K_t-K^n_t|^2\right]=0.
    \end{align*}
By Lemma \ref{estimateYn-S}, we have    
\begin{align*}
\E\left[\sup_{t\in[0,T]}|(\E[Y_t-L_t|\mathcal{G}_t])^-|^2 \right]=0, \  \E\left[\sup_{t\in[0,T]}|(\E[Y_t-U_t|\mathcal{G}_t])^+|^2 \right]=0,
\end{align*}
which implies that $\E[L_t|\mathcal{G}_t]\leq \E[Y_t|\mathcal{G}_t]\leq \E[U_t|\mathcal{G}_t]$, for any $t\in[0,T]$. It remains to prove the Skorokhod condition holds, which is similar with the one for the classical doubly reflected case (see Section 6 in \cite{EKPPQ}). So we omit it. The proof is complete. 
\end{proof}

\begin{remark}
    Compared to Theorem \ref{thm2.7}, the assumption made for filtration $\mathbb{G}$ is weaker in Theorem \ref{main2}.
\end{remark}

 \end{document}